\documentclass[12pt]{article}
\usepackage{amsfonts}
\usepackage{amsthm}
\usepackage{amsmath}
\usepackage[all]{xy}
\usepackage{titlesec}
\usepackage{dsfont}
\usepackage{amssymb}
\usepackage{extarrows}
\usepackage{mathrsfs}
\usepackage{arydshln}
\usepackage{multirow}
\usepackage{diagbox}
\theoremstyle{theorem}
\newtheorem{theorem}{Theorem}[section]

\theoremstyle{definition}
\newtheorem{definition}[theorem]{Definition}

\theoremstyle{proposition}
\newtheorem{proposition}[theorem]{Proposition}

\theoremstyle{corollary}
\newtheorem{corollary}[theorem]{Corollary}

\theoremstyle{lemma}
\newtheorem{lemma}[theorem]{Lemma}

\theoremstyle{lemma}
\newtheorem{remark}[theorem]{Remark}

\theoremstyle{remark}

\theoremstyle{remark}
\newtheorem*{statement}{Statement}

\renewcommand{\thesection}{\hspace{0.01pt}\arabic{section}}
\titleformat{\section}[hang]{\normalfont\normalsize\filright\bf}{\thesection}{0.5em}{}
\titlespacing{\section}{0em}{*0.5}{\wordsep}
\begin{document}

\title{The stable homotopy classification of $(n-1)$-connected $(n+4)$-dimensional polyhedra with  2 torsion free homology}
\author{\small{Jian Zhong PAN and Zhong Jian ZHU}}
\date{}

\maketitle
{\noindent\small{\bf{Abstract}}\quad
In this paper, we study the stable homotopy types of $\mathbf{F}^4_{n(2)}$-polyhedra, i.e., $(n-1)$-connected, at most $(n+4)$-dimensional polyhedra with 2-torsion free homologies. We are able to classify the indecomposable  $\mathbf{F}^4_{n(2)}$-polyhedra. The proof relies on the matrix problem technique which was developed in the classification of representaions of algebras and  applied to homotopy theory by Baues and Drozd.
\

{\noindent\small{\bf{keywords}}\quad
Homotopy; indecomposable; matrix problem}

\

\
\
\newcommand{\he}[1]{\ensuremath{\textbf{#1}}}
\newcommand{\s}[1]{\ensuremath{\scriptstyle{#1}}}
\newcommand{\hua}[1]{\ensuremath{\mathcal{#1}_{0}}}
\newcommand{\mo}[2]{\ensuremath{M_{3^{#1}}^{#2}}}
\section{Introduction}
\label{intro}

Let the $\mathbf{A}_n^k(n\geq k+1)$ be the subcategories of the stable homotopy category consisting of $(n-1)$-connected polyhedra with dimension at most $n+k$. It is a fully additive category if we consider the wedge of two polyhedra as the coproduct of two objects in the category $\mathbf{A}_n^k$. The classification problem of  $\mathbf{A}_n^k(n\geq k+1)$  is to find a complete list of indecomposable isomorphic classes, i.e. the indecomposable homotopy types in  $\mathbf{A}_n^k(n\geq k+1)$. For $k\leq3$, all indecomposable stable homotopy types have been described in \cite{RefBH}. For $k\geq 4$, Drozd shows the classification problem is wild (in the sense similar to that in representation of finite dimensional algebras) in \cite{RefDrMTS} by finding a wild subcategory of $\mathbf{A}_n^4(n\geq 5)$  whose objects are polyhedra with 2-torsion homologies.

In another direction, Baues and Drozd also consider full subcategory $\mathbf{F}_n^k$ of $\mathbf{A}_n^k(n\geq k+1)$  consisting of polyhedra with torsion free homology groups. For $k\leq 5$, such polyhedra have been classified to have finite indecomposable homotopy types in \cite{RefTF5cels}, \cite{RefTF6cels} or \cite{RefDrMSP}, \cite{RefDrMTS}. For $k=6$, Drozd got tame type classification of congruence classes of homotopy types, and proved that, for $k>6$, this problem is wild in \cite{RefDrFP}.

This is the second of a series of papers devoted to the homotopy theory of  $\mathbf{A}_n^k$-polyhedron.  In our previous paper \cite{PZ}, we noticed that, for $(n-1)$-connected and at most $(n+k)$-dimensional $(k<7)$ spaces with 2 and 3-torsion free homologies, the  classification of indecomposable stable homotopy types essentially reduces
to that of spaces with torsion free homologies. When homologies of the spaces involved have 3-torsion, the reduction process doesn't lead to the matrix problem for spaces with torsion free homologies but to a matrix problem which can be solved. By this we are able to classify homotopy types of the full subcategory  $\mathbf{F}_{n(2)}^4$ of $\mathbf{A}_n^4(n\geq 5)$  consisting of polyhedra with 2-torsion free homology groups. We will discuss the splitting of smash product of $\mathbf{A}_n^k$-polyhedra in a latter publication.

 Section 2 contains some basic notations and facts about stable homotopy category and classification problem. Our main theorem is given at the end of this section. In Section 3, Theorem \ref{theorem3.2} and Corollary \ref{corollary3.3} establish a connection between bimodule categories and stable homotopy categories. In Section 4, we use the known results of indecomposable homotopy types of $\mathbf{F}_{n}^4$ in \cite{RefTF5cels} to classify the indecomposable isomorphic classes of another matrix problem $(\mathscr{A}^{0},\mathcal{G}^0)$ corresponding to $\mathbf{F}_{n}^4$. In Section 5, the matrix problem $(\mathscr{A}',\mathcal{G}')$ used to classify the indecomposable homotopy types of $\mathbf{F}_{n(2)}^4$ is given. In Section 6, we solve the matrix problem $(\mathscr{A}',\mathcal{G}')$  by the results of indecomposable isomorphic classes of matrix problem $(\mathscr{A}^{0},\mathcal{G}^0)$.

\section{Preliminaries}
\label{sec:2}

In this paper ``Polyhedron'' is used as ``finite CW-complex'' and  ``Space'' means a based space; we denote by $\ast_{X}$ (or by $\ast$ if there is no ambiguity) the based point of the space $X$.
Denote by $Hot(X,Y)$ the set of homotopy classes of continuous maps $X\rightarrow Y$ and by $\textsf{CW}$ the homotopy category of polyhedra. The suspension functor  $\Sigma : X\mapsto X[1] ~(X[n]=\Sigma^{n}X)$ defines a natural map
 $Hot(X,Y)\rightarrow Hot(X[n],Y[n])$. Set $Hos(X,Y)$  $=\lim\limits_{n\rightarrow{\infty}}$ $ Hot(X[n],Y[n])$.
 If $\alpha\in Hot(X[n],Y[n]),~\beta\in Hot(Y[m],Z[m])$, the class $\beta[n]\cdot \alpha[m]\in Hot(X[m+n],Z[m+n])$ after stabilization is, by definition, the product $\beta\alpha$ of the classes of $\alpha$ and $\beta$ in $Hos(X,Z)$.
 Thus we obtain the stable homotopy category of polyhedra $\textsf{CWS}$. Extending $\textsf{CWS}$ by adding formal negative shifts $X[-n] (n\in \mathds{N})$ of  polyhedra and setting $Hos(X[-n],Y[-m]):=Hos(X[m],Y[n])$, one gets the category $\textsf{S}$ of~\cite{RefCo}, which is a fully additive category, and we denote it by $\textsf{CWS}$ too.

  We will say a polyhedron $X$ $p$-torsion free if all homology groups of $X$ are $p$-torsion free, where $p$ is a prime.
  Denote by $\mathbf{A}_{n}^{k}$ the full subcategory of $\textsf{CW}$ consisting of $(n-1)$-connected and at most $(n+k)$-dimensional polyhedra, and  denote by $\mathbf{F}_{n}^k$ (resp. $\mathbf{F}_{n(2)}^{k}$) the full subcategory of $\mathbf{A}_{n}^k$ consisting of torsion free (resp. 2-torsion free) polyhedra.   The suspension gives a functor $\Sigma :\mathbf{F}_{n}^{k}\rightarrow \mathbf{F}_{n+1}^{k}$ (resp. $\Sigma :\mathbf{F}_{n(2)}^{k}\rightarrow \mathbf{F}_{n+1(2)}^{k}$).
  By the Freudenthal Theorem(\cite{RefRMS} Theorem.6.26), it follows that
\begin{proposition}\label{proposition2.1}
If $dimX\leq d $ and $Y$ is $(n-1)$-connected,where $d<2n-1$, then the map $Hot(X,Y)\rightarrow Hot(X[1],Y[1])$ is bijective. If $d=2n-1$, this map is surjective. In particular, the map $Hot(X[m],Y[m])$ $\rightarrow$ $Hos(X,Y)$ is bijective if $m>d-2n+1$ and surjective if $m=d-2n+1$.
\end{proposition}
From  Proposition~\ref{proposition2.1}, we get the following

\begin{proposition}\label{proposition2.2}
The suspension functor induces equivalences $\mathbf{A}^{k}_{n} \xlongrightarrow{\sim} \mathbf{A}^{k}_{n+1}$ for all $n > k+1$. Moreover, if $n= k+1$, the suspension functor
$\mathbf{A}^{k}_{n} \xlongrightarrow{} \mathbf{A}^{k}_{n+1}$ is a full representation equivalence, i.e. it is full, dense and reflects isomorphisms.
\end{proposition}

 Note. If an additive  functor $F: \mathcal{C}\rightarrow \mathcal{D}$  is a full representation equivalence, denoted by $ \mathcal{C}\xlongrightarrow{F \simeq_{rep}}  \mathcal{D}$, then it induces an 1-1 correspondence of indecomposable isomorphic classes of objects of these two additive categories.

\begin{corollary}\label{corollary2.3}
 Functors $\Sigma :\mathbf{F}_{n}^{k}\rightarrow \mathbf{F}_{n+1}^{k}$ and $\Sigma :\mathbf{F}_{n(2)}^{k}\rightarrow \mathbf{F}_{n+1(2)}^{k}$ are equivalences of categories for $n\geq k+2$ and full representation equivalences for $n=k+1$.
\end{corollary}
Therefore $\mathbf{F}^{k}:= \mathbf{F}_{n}^{k}$ and $\mathbf{F}_{(2)}^{k}:= \mathbf{F}_{n(2)}^{k}$ with $n\geq k+2$ dose not depend on $n$.

 Let $\mathcal{C}$ be an additive category with zero object $\ast$ and biproducts $A\oplus B$ for any objects $A,B\in \mathcal{C}$, where $X\in \mathcal{C}$ means that $X$ is an object of $\mathcal{C}$. $X\in\mathcal{C} $ is decomposable if there is an isomorphism $X\cong A\oplus B$ where $A$ and $B$ are not isomorphic to $\ast$, otherwise $X$ is indecomposable. For example, $X\in \textsf{CW}$ (resp. $\textsf{CWS}$) is indecomposable if $X$ is homotopy equivalent (resp. stable homotopy equivalent ) to $X_{1} \vee X_{2} $ implies one of $X_{1}$ and $X_{2}$ is contractible. A decomposition of $X\in \mathcal{C}$ is an isomorphism
 $$X\cong A_{1}\oplus \cdots \oplus A_{n}, ~~n<\infty, $$
 where $A_{i}$ is indecomposable for $i\in \{1,2,\cdots , n\}$. The classification problem of category $\mathcal{C}$ is to find a complete list of indecomposable isomorphism types in $\mathcal{C}$ and describe the possible decompositions of objects in $\mathcal{C}$.

\begin{theorem}(\textbf{Main theorem})\label{maintheorem}
 The complete list of indecomposable (stable) homotopy types in $\mathbf{F}_{n(2)}^4$ $(n\geq 5)$ is given by the polyhedra in Theorem \ref{theorem6.3}; the Moore spaces $M(\mathbb{Z}/p^r, n)$,  $M(\mathbb{Z}/p^r, n+1)$,  $M(\mathbb{Z}/p^r, n+2)$, $M(\mathbb{Z}/p^r, n+3)$ where prime $p\neq 2$, $r\in \mathbb{N}_+$ and  $S^{n}$, $S^{n+1}$, $S^{n+2}$, $S^{n+3}$, $S^{n+4}$, $C_{\eta}=S^{n}\cup_{\eta}e^{n+2}$, where $\eta$ is the Hopf map.
\end{theorem}

\section{Techniques}
\label{sec:3}

\begin{definition}\label{definition3.1}
Let $\mathcal{A}$ and $\mathcal{B}$ be additive categories. $\mathcal{U}$ is an $\mathcal{A}$-$\mathcal{B}$-bimodule,
i.e. a biadditive functor  $\mathcal{A}^{op}\times \mathcal{B}\rightarrow \textbf{Ab}$, the category of abelian groups.
We define the bimodule category $El(\mathcal{U})$ as follows:
\begin{itemize}
  \item  the set of objects is the disjoint union $\bigcup_{A\in \mathcal{A} ,B\in \mathcal{B}}\mathcal{U}(A,B)$.
  \item A morphism $\alpha\rightarrow\beta$, where $\alpha \in \mathcal{U}(A,B), \beta \in \mathcal{U}(A',B')$ is a pair of morphisms $f:A\rightarrow A' $,
  $g:B\rightarrow B'$ such that $g\alpha=\beta f\in \mathcal{U}(A,B')$ (We write $g\alpha$ instead of $\mathcal{U}(1,g)\alpha$ and $\beta f$ instead of $\mathcal{U}(f,1)\beta$).
\end{itemize}
\end{definition}
Obviously $El(\mathcal{U})$ is an  (full) additive category if so are  $\mathcal{A}$ and $\mathcal{B}$.
\\

Suppose $\mathcal{A}$ and $\mathcal{B}$ are two full subcategories of $\textsf{CW}$ (or $\textsf{CWS}$), then we denote by $\mathcal{A}\dag \mathcal{B}$  the full subcategory of $\textsf{CW}$ (or  $\textsf{CWS}$) consisting of cofibers of  maps  $f:A\rightarrow B$, where $A\in \mathcal{A}, B\in \mathcal{B}$. We also denote by $\mathcal{A}\dag_{m} \mathcal{B}$  the full subcategory of $\mathcal{A}\dag \mathcal{B}$  consisting of cofibers of $f:A\rightarrow B$ such that $H_{m}(f)=0$ and denote by $\Gamma(A,B)$ the subgroup of $Hos(A,B)$ consisting of maps $f:A\rightarrow B$ such that $H_{m}(f)=0$, where $A\in \mathcal{A}$, $B\in \mathcal{B}$.

\begin{theorem}\label{theorem3.2}
 Let $\mathcal{A}$ and $\mathcal{B}$ be two full subcategories of $\textsf{CWS}$, suppose that  $ Hos$ $ (B, A[1])=0$ for all $A\in\mathcal{A}, B\in\mathcal{B}$. Consider  $H: \mathcal{A}^{op}\times \mathcal{B}\rightarrow \textbf{Ab}$, i.e.  $(A,B)\mapsto Hos(A,B)$, as an $\mathcal{A}$- $\mathcal{B}$-bimodule. Denote by $\mathcal{I}$ the ideal of category $\mathcal{A}\dag \mathcal{B}$ consisting of morphisms which factor both through $\mathcal{B}$ and $\mathcal{A}[1]$, and by $\mathcal{J}$ the ideal of the category $El(H)$ consisting of morphisms $(\alpha, \beta):f\rightarrow f'$ such that $\beta$ factors through $f'$ and $\alpha$ factors through $f$. Then
 \begin{itemize}
 \item [(1)] the functor   $ C:  El(H)\rightarrow  \mathcal{A}\dag \mathcal{B}$ ($f\mapsto C_{f}$) induces an equivalence  $El(H)/\mathcal{J}\simeq \mathcal{A}\dag \mathcal{B}/\mathcal{I}$.
 \item [(2)]  moreover $\mathcal{I}^{2}=0$, hence the projection $\mathcal{A}\dag\mathcal{B}\rightarrow \mathcal{A}\dag \mathcal{B}/\mathcal{I}$ is a representation equivalence.
 \item [(3)] In particular, let $n<m\leq n+k$ and denote by $\mathcal{\widetilde{B}}$ the full subcategory of $\mathbf{F}_{n(2)}^{k}(n\geq k+1)$ consisting of all $(n-1)$-connected polyhedra of dimension at most $m$ and by $\mathcal{\widetilde{A}}$ the full subcategory of $\mathbf{F}_{n(2)}^{k}(n\geq k+1)$ consisting of all $(m-1)$-connected polyhedra of dimension at most $n+k-1$., then
      $$El(H)/\mathcal{J}\xlongrightarrow{\overline{C}~\simeq}\mathcal{\widetilde{A}}\dag \mathcal{\widetilde{B}}/\mathcal{I}\xlongleftarrow{P\simeq_{rep}}\mathcal{\widetilde{A}}\dag \mathcal{\widetilde{B}}.$$
      gives a natural one-to-one correspondence between isomorphic classes of objects of $El(H)/\mathcal{J}$ and $\mathcal{\widetilde{A}}\dag\mathcal{\widetilde{B}}$.
      $\mathbf{F}_{n(2)}^{k}$ is the full subcategory of $\mathcal{\widetilde{A}}\dag\mathcal{\widetilde{B}}$ consisting of 2-torsion free polyhedra.
 \end{itemize}
 \end{theorem}
  \begin{proof}
  (1) and (2) of Theorem \ref{theorem3.2} follow directly from Theorem~1.1. of \cite{RefDrMTS}. It remains to show that $\mathbf{F}_{n(2)}^{k}$ is a full subcategory of $\mathcal{\widetilde{A}}\dag\mathcal{\widetilde{B}}$. For any $X\in \mathbf{F}_{n(2)}^{k}$, let $B=X^{n+2}$ be the $(n+2)$-skeleton of $X$. We get a cofiber sequence $B\rightarrow X\rightarrow X/B$. Since $X/B\simeq A[1]$ for some $A$ by Proposition  \ref{proposition2.2},
there is a cofiber sequence $A\xrightarrow {f} B\rightarrow X\rightarrow X/B$, i.e. $X\simeq C_{f}$. By the homology exact sequence of cofiber sequence, it is easy to know that $A\in \mathcal{\widetilde{A}}, B\in \mathcal{\widetilde{B}}$.
  \end{proof}

The following corollary follows from the Corollary 1.2 of \cite{RefDrMTS}.

\begin{corollary}\label{corollary3.3}
Under conditions of Theorem \ref{theorem3.2}, let $H_{0}$ be an $\mathcal{A}$-$\mathcal{B}$-subbimodule of $H$ such that $f_1af_2=0$ whenever $a\in El(H_{0})$, $f_{i}\in Hos(B_{i},A_{i}) (i=1.2)$. Denote by $\mathcal{A}\dag_{H_{0}}\mathcal{B}$ the full subcategory of $\mathcal{A}\dag\mathcal{B}$ consisting of cofibers of $a\in El(H_{0})$.
$\mathcal{I}_{H_{0}}=Mor(\mathcal{A}\dag_{H_{0}}\mathcal{B})\cap\mathcal{I}$ and $\mathcal{J}_{H_{0}}=Mor(El(H_{0}))\cap\mathcal{J}$. Then we have

 \begin{itemize}
   \item [(1)] $\mathcal{J}^2_{H_0}=\mathcal{I}_{H_{0}}^2=0$;
   \item [(2)] $C: El(H_{0})\xlongrightarrow{P \simeq_{rep}}El(H_{0})/\mathcal{J}_{H_{0}}\xlongrightarrow{\overline{C}~\simeq}\mathcal{A}\dag_{H_{0}} \mathcal{B}/\mathcal{I}_{H_{0}}\xlongleftarrow{P \simeq_{rep}}\mathcal{A}\dag_{H_{0}} \mathcal{B}.$
 \end{itemize}
\end{corollary}

If $H_{0}=\Gamma:\mathcal{A}^{op}\times \mathcal{B}\rightarrow \textbf{Ab}$, then $\mathcal{A}\dag_{H_{0}}\mathcal{B}=\mathcal{A}\dag_{m} \mathcal{B}$.

\textbf{Matrix problem} ~ Let $\mathscr{A}$ be a set of matrices which is closed under finite  direct sums of matrices  and let $\mathcal{G}$ denote the set of admissible transformations on $\mathscr{A}$.
 We say $A\cong B$ in $\mathscr{A}$ if $A$ can be transformed to $B$ by admissible transformations, and we say $A$ is decomposable
 if $A\cong A_{1}\bigoplus A_{2}$ for nontrivial $A_{1}, A_{2}\in \mathscr{A}$. The block matrices $\left( \begin{array}{c} A_{1} \\ 0 \\ \end{array} \right)$ and
$ \left( \begin{array}{cc}  A_{1} &0 \\  \end{array}  \right)$ are also thought to be decomposable. The matrix problem $(\mathscr{A}, \mathcal{G})$,
 or simply $\mathscr{A}$, means to classify the indecomposable isomorphic classes of $\mathscr{A}$  (denoted by $ind\mathscr{A}$) under admissible transformations $\mathcal{G}$.
Matrix problem $(\mathscr{A}, \mathcal{G})$ is said to be equivalent to matrix problem $(\mathscr{A}', \mathcal{G}')$ if there is a bijective map
$\varphi: \mathscr{A} \rightarrow  \mathscr{A}' $ such that $ A\cong A' $ in $ \mathscr{A} $ if and only if $ \varphi(A)\cong \varphi ( A')$ in $\mathscr{A}' $
and $\varphi(A_{1}\bigoplus A_{2}) = \varphi(A_{1})\bigoplus \varphi(A_{2})$. It is clear that if two matrix problems are equivalent, then there is a  one-to-one correspondence between their indecomposable isomorphic classes.

\begin{definition}\label{definition3.4}
Let $\mathscr{A}$ be a set of some matrices, ``$\cdot$'' is a ``product'' of two matrices defined in $\mathscr{A}$ ( ``$\cdot$'' may be not the usual matrix product),  we say that $M\in \mathscr{A} $ is invertible in $\mathscr{A}$ if there is a matrix $N\in \mathscr{A}$ such that $M\cdot N=N\cdot M=I\in \mathscr{A}$, where $I$ is the identity matrix.
\end{definition}

In the following context, for a matrix problem $(\mathscr{A}, \mathcal{G})$, saying a matrix $M\in \mathscr{A}$ invertible always means that $M$  is invertible in $\mathscr{A}$.

\section{The solution of a new matrix problem of the category $\mathbf{F}_{n}^4~(n\geq 5)$ }
\label{sec:4}

In the following context, the tabulations $\begin{tabular}{|c|c|c|}  \hline * &*  &*  \\ \hline * & * & *\\ \hline  \end{tabular}$ represent the matrices or block matrices.  For any category $\mathcal{C}$, denote by $ind\mathcal{C}$ the set of indecomposable isomorphic classes of $\mathcal{C}$.

 $ind \mathbf{F}_{n}^4 $ is known in \cite{RefTF5cels} and Drozd got a matrix problem corresponding to $\mathbf{F}_{n}^4$ in \cite{RefDrMTS}. Here we need a new matrix problem $(\mathscr{A}^{0}, \mathcal{G}^{0})$ for the classification problem of $\mathbf{F}_{n}^4 $.

When $n\geq 5$, denote by $\mathcal{B}_{0}$ the full subcategory of $\mathbf{F}_{n}^{4} (n\geq 5)$ consisting of all $(n-1)$-connected polyhedra of dimension at most $n+2$ and by $\mathcal{A}_{0}$ the full subcategory of $\mathbf{F}_{n}^{4}(n\geq 5)$ consisting of all $(n+1)$-connected polyhedra of dimension at most $n+2$. Then $\mathbf{F}_{n}^{4}= \mathcal{A}_{0}\dag_{n+2} \mathcal{B}_{0}$. From \cite{RefChang} we know
 $$ind\mathcal{A}_0=\{S^{n+2},S^{n+3}\};~~~~
ind\mathcal{B}_0=\{S^{n},S^{n+1},S^{n+2},C_{\eta}=S^{n}\cup _{\eta}e^{n+2}\}.$$

   Now take $m=n+2$, we obtain the $\mathcal{A}_{0}$-$\mathcal{B}_{0}$ subbimodule $\Gamma$ of $H$ :
    $$\Gamma : \mathcal{A}_{0}^{op}\times \mathcal{B}_{0}\rightarrow \textbf{Ab} ~~(A,B)\mapsto \Gamma(A,B),$$
  where $\Gamma(A,B)$ is the subgroup of $Hos(A,B)$ defined on section \ref{sec:3}.
 Take $H_{0}=\Gamma$ in Corollary \ref{corollary3.3}, then $f_{1}af_{2}=0$ whenever $f_{i}\in Hos(B_{i},A_{i})$ $(i=1,2)$, $a\in \Gamma(A_{2},B_{1})$, $A_{i}\in\mathcal{A}_{0}$ and $B_{i}\in \mathcal{B}_{0}$. Hence by Corollary \ref{corollary3.3}, we have
 $$C: El(\Gamma)\xlongrightarrow{P \simeq_{rep}}El(\Gamma)/\mathcal{J}_{\Gamma}\xlongrightarrow{\overline{C}~\simeq}\mathcal{A}_{0}\dag_{n+2} \mathcal{B}_{0}/\mathcal{I}_{\Gamma}\xlongleftarrow{P \simeq_{rep}}\mathcal{A}_{0}\dag_{n+2}\mathcal{B}_0.$$

     Objects of $El(\Gamma)$ can be represented by
    $5\times 2$ block matrices $(\gamma_{ij})$, where block $\gamma_{ij}$ has entries from the $(ij)$-th cell of $Table~ 1$. Morphisms $\gamma\rightarrow \gamma'$
     are given by block matrices $\alpha=(\alpha_{ij})_{2\times 2}$, $\beta=(\beta_{ij})_{5\times 5}$, $\alpha_{ij}$ has entries from the $(ij)$-th cell of $Table~2$ and $\beta_{ij}$  has entries from the $(ij)$-th cell of $Table~3$. Their sizes are compatible with those of $\gamma_{ij}$ and  $\gamma_{ij}'$ and $\beta\gamma=\gamma'\alpha$.
     Such a morphism is invertible if and only if $\alpha$ and $\beta$ are invertible in $Hos(\hua{A},\hua{A})$ and $Hos(\hua{B},\hua{B})$ respectively. And it is equivalent to say that all diagonal blocks of $\alpha$ and $\beta$ are square, and both $det(\alpha)$, $det(\beta)$ equal to $\pm 1$.
     Since only entries from $\mathbb{Z}$ and $2\mathbb{Z}$ give nonzero input to the determinants, they belong indeed to $\mathbb{Z}$.
We get the corresponding matrix problem of $El(\Gamma)$ which denoted by $(\mathscr{A}^{0}, \mathcal{G}^{0})$.

$$\begin{tabular}{|c|c|c|}
\multicolumn{3}{c}{$\Gamma(\hua{A},\hua{B})$}\\
 \multicolumn{3}{c}{}\\
\hline
\diagbox{\hua{B}}{\hua{A}} & $S^{n+2}$ & $S^{n+3}$  \\
\hline
 $S^{n}$ & $\mathbb{Z}/2$ &$\mathbb{Z}/24$  \\
\hline
$S^{n+1}$ &$\mathbb{Z}/2$& $\mathbb{Z}/2$ \\
\hline
$S^{n+2}$ & 0& $\mathbb{Z}/2$\\
\hline
$C_{\eta}:n$ &0 & $\mathbb{Z}/12$ \\
\qquad n+2 &0  & 0   \\
\hline
\end{tabular}$$
$$Table~1$$

$$\begin{tabular}{|c|c|c|}
 \multicolumn{3}{c}{$Hos(\hua{A},\hua{A})$}\\
 \multicolumn{3}{c}{}\\
\hline
\diagbox{\hua{A}}{\hua{A}}& $S^{n+2}$ & $S^{n+3}$  \\
\hline
 $S^{n+2}$ & $\mathbb{Z}$ &$\mathbb{Z}/2$  \\
\hline
$S^{n+3}$ &0& $\mathbb{Z}$ \\
\hline
\end{tabular}$$
$$Table~2$$

$$\begin{tabular}{|c|c|c|c|cc|}
 \multicolumn{6}{c}{$Hos(\hua{B},\hua{B})$}\\
 \multicolumn{6}{c}{}\\
\hline
\diagbox{\hua{B}}{\hua{B}}& $S^{n}$ & $S^{n+1}$ &$S^{n+2}$ &$C_{\eta}:n$ & n+2 \\
\hline
 $S^{n}$ & $\mathbb{Z}$ &$\mathbb{Z}/2$ &$\mathbb{Z}/2$ &$2\mathbb{Z} $&0\\
\hline
$S^{n+1}$ &0& $\mathbb{Z}$ &$\mathbb{Z}/2$ & 0 &0\\
\hline
$S^{n+2}$ &0& 0 &$\mathbb{Z}$ & 0 &$\mathbb{Z}$\\
\hline
$C_{\eta}:n$ &$\mathbb{Z}$ &0&$\mathbb{Z}/2^=$& $\mathbb{Z}^{=}$&0\\
\quad n+2 &0 & 0 & $2\mathbb{Z}^=$ & 0 & $\mathbb{Z}^{=}$ \\
\hline
\end{tabular}$$
$$Table~3$$
In $Table~3$, $Hos(C_{\eta},C_{\eta})$ is identified with the ring

$$\left(
\begin{array}{cc}
 \mathbb{Z}^{=}& 0 \\0 & \mathbb{Z}^{=}\\
 \end{array} \right)
 = \left\{ \left. \left( \begin{array}{cc}
                                  a & 0 \\
                                  0 & b \\
                                \end{array}
           \right) \right |a\equiv b~(mod 2) \right\} ;$$
$Hos(S^{n+2},C_{\eta})$ is identified with the subgroup $\left(
\begin{array}{c}
\mathbb{Z}/2^{=}\\
 2\mathbb{Z}^{=} \\
 \end{array}
\right)$  of
   $$ \left(
                              \begin{array}{c}
                                \mathbb{Z}/2\\
                                 2\mathbb{Z} \\
                                  \end{array}
                                    \right) = \left\{ \left. \left( \begin{array}{c}
                                  \varepsilon  \\
                                  2a \\
                                \end{array}
                              \right) \right | \varepsilon\in   \mathbb{Z}/2 , a\in \mathbb{Z} \right\}$$
 which is the image of  the following injective map
 $$Hos(S^{n+2},C_{\eta})\xrightarrow {\mathcal{F}}\left(
 \begin{array}{c}
  \mathbb{Z}/2\\
  2\mathbb{Z} \\
  \end{array}
  \right)   ~~~~f\mapsto \left( \begin{array}{c}
                                  \varepsilon  \\
                                  2a \\
                                \end{array}
                              \right).$$
  For any $f\in Hos(S^{n+2},C_{\eta})$, let $S^{n+1}\xrightarrow {\eta}S^{n}\xrightarrow {i}C_{\eta}\xrightarrow {q}S^{n+2}$ be the cofiber sequence  and $S^{n+3}$$\xrightarrow {\eta_{n+2}}$ $S^{n+2}$ be the suspension of $\eta$. Then  $qf=2a\iota_{n+2}\in Hos(S^{n+2},S^{n+2})\cong \mathbb{Z}$  for some $a\in \mathbb{Z}$, where $\iota_{n+2}:$ $S^{n+2}$ $\rightarrow$ $S^{n+2}$ is the identity map. Let $\varepsilon=\left\{
                                      \begin{array}{ll}
                                        1, & ~~~\hbox{if $f\eta_{n+2}\neq 0$} \\
                                        0, & ~~~\hbox{if $f\eta_{n+2}= 0$}
                                      \end{array}
                                    \right.$,
$\mathcal{F}$ is defined by mapping $f$ to $\left(
                                              \begin{array}{c}
                                                \varepsilon \\
                                                2a\\
                                              \end{array}
                                            \right)$.

In order to make the product of matrices in $Table~3$ compatible with the composition of the corresponding  maps, special rule for the matrix product in $Table~3$ is needed:
\begin{itemize}
 \item [(1)] For $\left(
                  \begin{array}{cc}
                    2a & 0 \\
                  \end{array}
                \right)$ and
 $\left(
  \begin{array}{c}
    1 \\
    2b \\
  \end{array}
  \right)$ respectively in
 \begin{footnotesize}$\begin{tabular}{c|cc|}
   \multicolumn{1}{c}{}&\multicolumn{2}{c} {$C_{\eta}:\s{n , n+2}$}\\
       \cline{2-3}
     $S^{n}$ & ~~$2\mathbb{Z}$& 0 \\
       \cline{2-3}
     \end{tabular}$\end{footnotesize}~ and ~
    \begin{footnotesize}$\begin{tabular}{c|c|}
   \multicolumn{1}{c}{}&\multicolumn{1}{c} {$S^{n+2}$}\\
       \cline{2-2}
      $C_{\eta}:\s{{n}}$ & $\mathbb{Z}/2^{=}$\\
      \quad ${\s{n+2}}$ &  $2\mathbb{Z}^{=}$ \\
        \cline{2-2}
      \end{tabular}$\end{footnotesize} ,
      \newline
   $\left(\begin{array}{cc}
                   2a & 0 \\
                  \end{array}  \right)\left(
  \begin{array}{c}
    1 \\
    2b \\
  \end{array}
  \right)$ $=\overline{a}$  in  $\begin{footnotesize}\begin{tabular}{c|c|}
   \multicolumn{1}{c}{}&\multicolumn{1}{c} {$S^{n+2}$}\\
       \cline{2-2}
     $S^{n}$ & $\mathbb{Z}/2$ \\
       \cline{2-2}
     \end{tabular}\end{footnotesize}$ , where  $\overline{a}$  is the image of $a$ under the quotient map $\mathbb{Z}\rightarrow \mathbb{Z}/2$.

  \item [(2)] For  $\left(
  \begin{array}{c}
    a \\
    0 \\
  \end{array}
  \right)$   and $\left(
                   \begin{array}{c}
                    \varepsilon \\
                   \end{array}
                 \right)$ respectively in  \begin{footnotesize}$\begin{tabular}{c|c|}
   \multicolumn{1}{c}{}&\multicolumn{1}{c} {$S^{n}$}\\
       \cline{2-2}
      $C_{\eta}:\s{{n}}$ & $\mathbb{Z}$\\
      \quad ${\s{n+2}}$ &  $0$ \\
        \cline{2-2}
      \end{tabular}$\end{footnotesize} ~and  $\begin{footnotesize}\begin{tabular}{c|c|}
   \multicolumn{1}{c}{}&\multicolumn{1}{c} {$S^{n+2}$}\\
       \cline{2-2}
     $S^{n}$ & $\mathbb{Z}/2$ \\
       \cline{2-2}
     \end{tabular}\end{footnotesize}$ ,  $\left(
  \begin{array}{c}
    a \\
    0 \\
  \end{array}
  \right)\left(
                   \begin{array}{c}
                    \varepsilon \\
                   \end{array}
                 \right)=\left(
  \begin{array}{c}
    0\\
    0 \\
  \end{array}
  \right)$ in  \begin{footnotesize}$\begin{tabular}{c|c|}
   \multicolumn{1}{c}{}&\multicolumn{1}{c} {$S^{n+2}$}\\
       \cline{2-2}
      $C_{\eta}:\s{{n}}$ & $\mathbb{Z}/2^{=}$\\
      \quad ${\s{n+2}}$ &  $2\mathbb{Z}^{=}$ \\
        \cline{2-2}
      \end{tabular}$\end{footnotesize} .

  \item [(3)]  Keep elements in zero blocks being zero. For example, for any
$\left(
 \begin{array}{c}
  a \\
 \end{array}
\right)$ and
$\left(
 \begin{array}{cc}
  0& b \\
 \end{array}
 \right)$ respectively in
$\begin{footnotesize}\begin{tabular}{c|c|}
  \multicolumn{1}{c}{}&\multicolumn{1}{c} {$S^{n+2}$}\\
  \cline{2-2}
  $S^{n}$ & $\mathbb{Z}/2$ \\
  \cline{2-2}
  \end{tabular}\end{footnotesize}$ and
\begin{footnotesize}$\begin{tabular}{c|cc|}
  \multicolumn{1}{c}{}&\multicolumn{2}{c} {$C_{\eta}:\s{n , n+2}$}\\
  \cline{2-3}
 $S^{n+2}$ & ~~0&$\mathbb{ Z}$\\
  \cline{2-3}
 \end{tabular}$\end{footnotesize}~,
$\left(
 \begin{array}{c}
 a \\
 \end{array}
 \right)$
$\left(
 \begin{array}{cc}
  0& b \\
\end{array}
\right)$ $=$ $\left(
        \begin{array}{cc}
          0& 0 \\
        \end{array}
      \right)$ in
 \newline
 \begin{footnotesize}$\begin{tabular}{c|cc|}
   \multicolumn{1}{c}{}&\multicolumn{2}{c} {$C_{\eta}:\s{n , n+2}$}\\
       \cline{2-3}
     $S^{n+2}$ &$\mathbb{ Z}$& 0\\
       \cline{2-3}
     \end{tabular}$\end{footnotesize}~, the second element is not $ab$ but 0.
  \end{itemize}
Denote by $W_{x}$ (respectively $W^{y}$) the $x$- horizontal (respectively $y$-vertical) stripe, where
 $x\in \{S^{n},S^{n+1}$, $S^{n+2}$, $C_{\eta}:{\s{n}}$, $C_{\eta}:{\s{n+2}}\}$, $y\in \{S^{n+2},S^{n+3}\}$, and denote by $W^{y}_{x}$ the block corresponding to $x$- horizontal stripe and $y$- vertical stripe. Let dim $W_{x}=$ the number of rows in $W_{x}$, dim $W^{y}=$ the number of columns in $W^{y}$.
$Table~1$ represents the matrix set $\mathscr{A}^{0}$. By right multiplication with invertible matrices in $Table~2$ and left multiplication with invertible matrices in $Table~3$, $Table~2$ and $Table~3$ provide admissible transformations $ \mathcal{G}^{0}$ (see \cite{RefDrMSP}) for  matrices in $Table 1$, i.e.
\begin{itemize}
  \item [(a)] ``elementary-row transformations'' of $W_{x}$ consisting of following three types:
   \begin{itemize}
    \item [](j+$a$i)-type : The replacement of the $j$-th row $\alpha_{j}$ of $W_{x}$ by $\alpha_{j}+a\alpha_{i}$, where $\alpha_{i}$ is the $i$-th row of $W_{x}$ , $a\in \mathbb{Z}$.
    \item  []($a$i)-type : The multiplication of the $i$-th row $\alpha_{i}$ of $W_{x}$ by $a\in \{\pm 1\}$.
    \item  [](i,j)-type : The transposition of the $i$-th and $j$-th row.
   \end{itemize}
   \item[(b)] ``elementary-column transformations'' of $W^{y}$ which also have three types as for elementary-row transformations;
   \item [](Restriction on (a) and (b)) If one performs a (j+$a$i)-type (respectively ($a$i)-type and (i,j)-type) elementary-row transformation of $W_{C_{\eta}:n}$,
   then one has to perform the (j+$a'$i)-type (respectively ($a'$i)-type and (i,j)-type) elementary-row transformation of  $W_{C_{\eta}:n+2}$ simultaneously
    where $a\equiv a'~(mod 2)$ and vice versa;
   \item[(c)] Adding $k$ times of a column of $W^{S^{n+2}}$ to a column of $W^{S^{n+3}}$;
   \item[(d)] Adding $k$ times of a row of $W_{S^{n+1}}$ or $W_{S^{n+2}}$ to a row of $W_{S^{n}}$;
   \item[(e)] Adding $k$ times of a row of $W_{S^{n+2}}$ to a row of $W_{S^{n+1}}$;
   \item[(f)] \begin{itemize}
                \item [(1)]  Adding $k$ times of a row of $W_{S^{n}}$ to a row of $W_{C_{\eta}:n}$;
                \item [(2)] Adding $2k$ times of of a row of $W_{C_{\eta}:n}$ to a row of $W_{S^{n}}$;
              \end{itemize}
   \item[(g)] Adding $6k$ times of a row of $W_{S^{n+2}}$ to a row of $W_{C_{\eta}:n}$;
   \end{itemize} where $k$ is an integer.

  \begin{remark}
    When admissible transformations above are performed on block matrix $\gamma=(\gamma_{ij})$, where block $\gamma_{ij}$ has entries from (ij)-cell of $table~ 1$,  we should note that
  \begin{itemize}
    \item [(1)]  If (ij)-cell of $table~ 1$ is zero, then $\gamma_{ij}$ keeps being zero after admissible transformations;
    \item [(2)]  Adding $1\in \mathbb{Z}/2$ to an element $a\in \mathbb{Z}/24$ gives $a+12$ in $\mathbb{Z}/24$, since $\eta^{3}$ is $12$ in $\mathbb{Z}/24 = Hos(S^{n+3},S^{n})$.
    \item [(3)] The reason for (g) is as follow: in the definition of the injective map $\mathcal{F}$ above, for any $f\in Hos(S^{n+2},C_{\eta})$, $f\eta = ix$ for some $x\in Hos(S^{n+3},S^{n})=\mathbb{Z}/24$. If $qf=2\iota_{n+2}\in Hos(S^{n+2},S^{n+2})$ then $x=6$ (Proposition~6 (iii) of \cite{RefUnsold}).
  \end{itemize}
  \end{remark}
From the known fact that
$$ind(\mathscr{A}^{0})\cong indEl(\Gamma)\cong ind(\mathcal{A}_{0}\dag_{n+2}\mathcal{B}_0)= ind\mathbf{F}_{n}^{4}.$$
 we have

\textbf{List (*)} :
\begin{itemize}
  \item [(I)] $X(\eta v\eta)=S^n\vee S^{n+2}\cup_{i_1\eta}e^{n+2}\cup_{i_{1}v+i_{2}\eta}e^{n+4}$  corresponds to
   $$\begin{tabular}{c|c|}
   \multicolumn{1}{c}{}&\multicolumn{1}{c} {$S^{n+3}$}\\
       \cline{2-2}
     $S^{n+2}$ & 1 \\
       \cline{2-2}
      $C_{\eta}:{\scriptstyle{n}}$ & $v$\\
      \quad ${\scriptstyle{n+2}}$ & 0 \\
        \cline{2-2}
     \end{tabular}
   $$  where $v\in \{1,2,3\}\subset\mathbb{Z}/12$.
  \item [(II)]
    \begin{itemize}
      \item [(1)] $X(\eta\eta v\eta\eta)=S^n\vee S^{n+1}\cup_{i_1\eta\eta}e^{n+3}\cup_{i_{1}v+i_{2}\eta\eta}e^{n+4}$ corresponds to  $$\begin{tabular}{c|c|c|}
   \multicolumn{1}{c}{}&\multicolumn{1}{c} {$S^{n+2}$}&\multicolumn{1}{c} {$S^{n+3}$}\\
       \cline{2-3}
     $S^{n}$ & 1& $v$ \\
        \cline{2-3}
      $S^{n+1}$ & 0& 1 \\
       \cline{2-3}
     \end{tabular}$$

     \item [(2)] $X(\eta\eta v\eta)=S^n\vee S^{n+2}\cup_{i_1\eta\eta}e^{n+3}\cup_{i_{1}v+i_{2}\eta}e^{n+4}$  corresponds to $$\begin{tabular}{c|c|c|}
   \multicolumn{1}{c}{}&\multicolumn{1}{c} {$S^{n+2}$}&\multicolumn{1}{c} {$S^{n+3}$}\\
       \cline{2-3}
     $S^{n}$ & 1& $v$ \\
        \cline{2-3}
      $S^{n+2}$ & 0& 1 \\
       \cline{2-3}
     \end{tabular}$$

       \item [(3)] $X(\eta v\eta\eta)= S^n\vee S^{n+1}\cup_{i_1\eta}e^{n+2}\cup_{i_{1}v+i_{2}\eta}e^{n+4}$ corresponds to $$\begin{tabular}{c|c|}
   \multicolumn{1}{c}{}&\multicolumn{1}{c} {$S^{n+3}$}\\
       \cline{2-2}
     $S^{n+1}$ & 1 \\
       \cline{2-2}
      $C_{\eta}:{\scriptstyle{n}}$ & $v$\\
      \quad ${\scriptstyle{n+2}}$ & 0 \\
        \cline{2-2}
     \end{tabular} $$

     \item [(4)] $X(\eta\eta v)=S^n\cup_{\eta\eta}e^{n+3}\cup_{v}e^{n+4}$ corresponds to
      $$\begin{tabular}{c|c|c|}
      \multicolumn{1}{c}{}&\multicolumn{1}{c} {$S^{n+2}$}&\multicolumn{1}{c} {$S^{n+3}$}\\
       \cline{2-3}
     $S^{n}$ & 1& $v$ \\
      \cline{2-3}
     \end{tabular}$$

     \item [(5)] $X( v\eta\eta)=S^n\vee S^{n+1}\cup_{i_{1}v+i_{2}\eta\eta}e^{n+4}$ corresponds to
     $$\begin{tabular}{c|c|}
   \multicolumn{1}{c}{}&\multicolumn{1}{c} {$S^{n+3}$}\\
       \cline{2-2}
     $S^{n}$ & $v$ \\
       \cline{2-2}
      $S^{n+1}$ & $1$\\
        \cline{2-2}
     \end{tabular} $$

      \item [(6)] $X(\eta v)=S^n\cup_{\eta}e^{n+2}\cup_{v}e^{n+4}$  corresponds to
      $$\begin{tabular}{c|c|}
   \multicolumn{1}{c}{}&\multicolumn{1}{c} {$S^{n+3}$}\\
       \cline{2-2}
   $C_{\eta}:{\scriptstyle{n}}$ & $v$\\
      \quad ${\scriptstyle{n+2}}$ & 0 \\
        \cline{2-2}
     \end{tabular} $$

       \item [(7)] $X(v\eta )=S^n\vee S^{n+2}\cup_{i_{1}v+i_{2}\eta}e^{n+4}$ corresponds to
      $$\begin{tabular}{c|c|}
   \multicolumn{1}{c}{}&\multicolumn{1}{c} {$S^{n+3}$}\\
       \cline{2-2}
     $S^{n}$ & $v$ \\
       \cline{2-2}
      $S^{n+2}$ & $1$\\
        \cline{2-2}
     \end{tabular} $$
    \end{itemize}
  where $v\in\{1,2,3,4,5,6\}\subset \mathbb{Z}/24 $  in the cases (1),(2),(4),(5),(7) of (II), and $v\in\{1,2,3,4,5,6\}$$\subset $$\mathbb{Z}/12 $ in the case (3),(6) of (II).
        \item [(III)] $X(v)=S^{n}\cup_{v}e^{n+4}$ corresponds to
         $$\begin{tabular}{c|c|}
   \multicolumn{1}{c}{}&\multicolumn{1}{c} {$S^{n+3}$}\\
       \cline{2-2}
     $S^{n}$ & $v$ \\
      \cline{2-2}
    \end{tabular} $$ where $v\in\{1,2,\cdots,12\}\subset \mathbb{Z}/24$.

         \item [(IV)]
         \begin{itemize}
          \item [(1)] $X(\eta_1)=S^{n+1}\cup_{\eta}e^{n+3}$ corresponds to
           $\begin{tabular}{c|c|}
   \multicolumn{1}{c}{}&\multicolumn{1}{c} {$S^{n+2}$}\\
       \cline{2-2}
     $S^{n+1}$ & $1$ \\
      \cline{2-2}
    \end{tabular}~; $

           \item [(2)] $X(\eta_2)=S^{n+2}\cup_{\eta}e^{n+4}$ corresponds to
           $\begin{tabular}{c|c|}
   \multicolumn{1}{c}{}&\multicolumn{1}{c} {$S^{n+3}$}\\
       \cline{2-2}
     $S^{n+2}$ & $1$ \\
      \cline{2-2}
    \end{tabular}~; $

           \item [(3)] $X(\eta\eta)_0 =S^{n}\cup_{\eta\eta}e^{n+3}$ corresponds to
                $\begin{tabular}{c|c|}
   \multicolumn{1}{c}{}&\multicolumn{1}{c} {$S^{n+2}$}\\
       \cline{2-2}
     $S^{n}$ & $1$ \\
      \cline{2-2}
    \end{tabular}~; $

            \item [(4)] $X(\eta\eta)_1 =S^{n+1}\cup_{\eta\eta}e^{n+4}$ corresponds to
                $\begin{tabular}{c|c|}
   \multicolumn{1}{c}{}&\multicolumn{1}{c} {$S^{n+3}$}\\
       \cline{2-2}
     $S^{n+1}$ & $1$ \\
      \cline{2-2}
    \end{tabular}~, $
         \end{itemize}
\end{itemize}
  For a wedge of spaces $X\vee Y$,  $i_{1}: X\hookrightarrow X\vee Y$ and  $i_{2}: X\hookrightarrow X\vee Y$  above are the canonical inclusions .

      \begin{remark}
  Indecomposable homotopy types in $\{A[1]~|~A\in ind\mathcal{A}_{0}\}$ and $ind \mathcal{B}_{0}$ of $\mathbf{F}_{n}^{4}$ are not contained in  $\textbf{List (*)}$. An element $A[1]$ of $\{A[1]~|~A\in ind\mathcal{A}_{0}\}$  (resp. $B$ of $ind \mathcal{B}_{0}$) can be considered as a mapping cone of map  $A\rightarrow \ast$ (resp. $\ast \rightarrow B$) in $\mathcal{A}_{0}\dag_{n+2}\mathcal{B}_0$ which corresponds to $0\times 1$ matrix (resp. $1\times 0$ matrix) in $\mathscr{A}^{0}$.  For a general matrix problem $(\mathscr{A}, \mathcal{G})$, these  $0\times 1$  and  $1\times 0$ matrices are regarded as elements in $ind\mathscr{A}$, but will not be listed to simplify notation.
      \end{remark}

\section{The reduction of the classification problem of $\mathbf{F}_{n(2)}^4$~$(n\geq 5)$}
\label{sec:5}

Let $M_{t}^{k}$ be the Moore space $M(\mathbb{Z}/t,k)$ ,  $t,k \in \mathbb{N}_+=\{1,2,\cdots,\}$.

  Take $m=n+2$ and two full subcategories
 $\widetilde{\mathcal{A}}$ and $\widetilde{\mathcal{B}}$ of $\mathbf{F}_{n(2)}^{4}$
 as in Theorem \ref{theorem3.2} (3).
 By the results of the indecomposable homotopy types of $\mathbf{A}_n^2~(n\geq3)$ in~\cite{RefChang},
we have
\\
\\
$ind\mathcal{\widetilde{A}}=\{S^{n+2},S^{n+3},M_{p^{r}}^{n+2}~|~ \text{ prime}~p\neq2, r\in \mathbb{N}_+\};$
\newline
$ind\mathcal{\widetilde{B}}=\{S^{n},S^{n+1},S^{n+2},C_{\eta}=S^{n}\cup_{\eta}e^{n+2} ,M_{p^{s}}^{n}, M_{p^{s}}^{n+1}~|~ \text{ prime}~p\neq2, s\in \mathbb{N}_+\}.$

\begin{lemma}\label{lemma5.1}
 \begin{itemize}
  \item []
  \item [] $Hos(M_{p^{r}}^{n+2}, B)=0$ ~for any $B\in ind\mathcal{\widetilde{B}}$, where prime $p\neq 2,3; r\in \mathbb{N}_+$;
   \item [] $Hos( A, M_{p^{s}}^{n} )=0$ ~for any $A\in ind\mathcal{\widetilde{A}}$, where prime $p\neq 2,3; s\in \mathbb{N}_+$;
    \item []  $Hos( A, M_{p^{s}}^{n+1} )=0$ ~for any $A\in ind\mathcal{\widetilde{A}}$, where prime $p\neq 2; s\in \mathbb{N}_+$.
  \end{itemize}
 \end{lemma}

\begin{proof}
 It follows from the triviality of $p$-primary component of relevant homotopy groups of spheres and the universal coefficients theorem for homotopy groups with coefficients.
\end{proof}

For $C_{f}\in \mathcal{\widetilde{A}}\dag \mathcal{\widetilde{B}}$,
 where $f:A\rightarrow B$ ,$A=\vee A_{i}, A_{i}\in ind\mathcal{\widetilde{A}}$, $B=\vee B_{j}$, $B_{j}\in ind\mathcal{\widetilde{B}}$. If $A_{i}=M_{p^{r}}^{n+2}~(p\neq 2,3)$
  for some $i$, then  $A_{i}[1]$ split out of $C_{f}$. Similarly if $B_{j}=M_{p^{s}}^{n}~(p\neq 2,3)$ or $M_{p^{s}}^{n+1}~(p\neq 2)$ for some $j$, then this $B_{j}$ also split out of $C_{f}$. So we get the following
 \begin{lemma}\label{lemma5.2}
 Let $\mathcal{A}$ and $\mathcal{B}$ be the full subcategories of $\mathcal{\widetilde{A}}$ and $\mathcal{\widetilde{B}}$ respectively, such that
  \begin{itemize}
  \item [] $ind\mathcal{A}=\{S^{n+2},S^{n+3},M_{3^{r}}^{n+2}~|~ r\in \mathbb{N}_+\};$
   \item [] $ ind\mathcal{B}=\{S^{n},S^{n+1},S^{n+2},C_{\eta}=S^{n}\cup_{\eta}e^{n+2} ,M_{3^{s}}^{n} ~|~ s\in \mathbb{N}_+\}.$
  \end{itemize}
 then $$ind(\mathcal{\widetilde{A}}\dag\mathcal{\widetilde{B}})=ind(\mathcal{A}\dag\mathcal{B})\cup \{M_{p^{r}}^{n+3}, M_{p^{r}}^{n}, M_{q^{r}}^{n+1} ~| ~\text{primes}~p\neq2,3, q\neq2; r\in \mathbb{N}_+\}.$$
 \end{lemma}

 By Theorem \ref{theorem3.2} (3),
  \begin{corollary}\label{corollary 5.3}
$ind\mathbf{F}_{n(2)}^4=\{X\in ind(\mathcal{A}\dag\mathcal{B})$ $|~X$ is $2$-torsion free
$\}$ $\cup$ $\{~M^{n+3}_{p^r}, M^{n}_{p^r},$ $ M^{n+1}_{q^r}~|~$ prime $p\neq 2,3$, prime $q\neq 2$ and $r\in \mathbb{N}_+\}.$
\end{corollary}

  In order to get $ind\mathbf{F}_{n(2)}^4$, it suffices to compute $\{X\in ind(\mathcal{A}\dag\mathcal{B})$ $|~X$ is $2$-torsion free
$\}$.

 Let $\Gamma : \mathcal{A}^{op}\times\mathcal{B}\rightarrow \textbf{Ab}$, $\Gamma(A,B)=\{g\in Hos(A,B)~|~H_{n+2}(g)=0\}$, defined in section \ref{sec:3}, be a sub-bimodule of $\mathcal{A}$-$\mathcal{B}$-bimodule $H: \mathcal{A}^{op}\times\mathcal{B}\rightarrow \textbf{Ab}$, $H(A,B)=Hos(A,B)$.

 \begin{lemma}\label{lemma5.4}
  $$\begin{array}{rl}
   & \{X\in ind(\mathcal{A}\dag\mathcal{B})~|~X~\text{is}~2\text{-torsion free}
   \} \\
   =&\{M_{p^{r}}^{n+2}~|~ \text{prime}~ p\neq 2, r\in \mathbb{N}_{+}\}\cup \{~C(f)~\text{is indecomposable}~|~f\in El(\Gamma)~\}.
  \end{array}$$
 \end{lemma}

\begin{proof}
For integers $k,l,t,u,v,w\geq0$, let
 $$A(k,l):=\bigvee_{k}S^{n+2}\vee \bigvee_{l}S^{n+3}\in \mathcal{A}_{0};$$
     $$B(t,u,v,w):=\bigvee_{t}S^{n+2}\vee\bigvee_{u}C_{\eta}\vee\bigvee_{v}S^{n}\vee\bigvee_{w}S^{n+1}\in \mathcal{B}_{0}.$$
For any 2-torsion free polyhedra $X=C_{f}\in \mathcal{A}\dag\mathcal{B}$,  $f\in Hos(A,B)$ where $A\in \mathcal{A}$, $B\in \mathcal{B}$,  suppose that
 $$A=A(k,l)\vee M_{A},~~~~ B=B(t,u,v,w)\vee M_{B},$$ where $M_{A}$ (resp. $M_{B}$) is a wedge of Moore spaces $\{\mo{r}{n+2}~|~r\in \mathbb{N}_{+} \}$ (resp. $\{\mo{s}{n}~|~s\in \mathbb{N}_{+} \}$ ).  Let
  $$\xymatrix{
    A(k,l)\ar@{^{(}->}[r]^{~~j_{A}} & A, & B  \ar@{->>}[r]^{p_{B}~~~~~~} & B(t,u,v,w)
   }$$ be the canonical inclusion and projection of the summands respectively.  For $$h:=p_{B}fj_{A}\in Hos(~A(k,l), ~ B(t,u,v,w)~),$$
  by the proof of Theorem 5.5 of \cite{PZ}, we have the commutable top square in the following Diagram 1, where $k_{1}+k'=k,  k_{1}+t'=t$, $\alpha$ and $\beta$  are  self-homotopy equivalences of $A(k,l)$ and $B(t,u,v,w)$ respectively
   and the maps $$\bigvee_{k_{1}}S^{n+2}\xlongrightarrow{h_1}\bigvee_{k_{1}}S^{n+2},~~~A(k',l)\xlongrightarrow{h'}B(t',u,v,w)$$ satisfy that
\begin{itemize}
  \item [(i)] the mapping cone
  $C_{h_{1}}=\bigvee_{i}M_{\alpha_{i}}^{n+2}$, where $\alpha_{i}\in \mathbb{N}_+$ is odd for each $i$;
  \item [(ii)]   the composition of maps
  $$\bigvee_{k'}S^{n+2}\xlongrightarrow{j}A(k',l)\xlongrightarrow{h'}B(t',u,v,w)\xlongrightarrow{p}\bigvee_{t'}S^{n+2}\vee\bigvee_{u}C_{\eta}$$
  is zero, where $j$ and $p$ are canonical inclusion and projection of the summands respectively. This is equivalent to the statement that $H_{n+2}(h')=0$.
  \end{itemize}
 Note that $Hos(S^{n+2}, M_{B})=0$ and $Hos(M_{A},S^{n+2})=0$. Hence
 $$(\beta\vee 1_{M_{B}})f(\alpha\vee 1_{M_{A}})=h_{1}\vee f'$$
 such that $A(k',l)\vee M_{A}\xlongrightarrow{f'}B(t',u,v,w)\vee M_{B}$ satisfies $H_{n+2}(f')=0$.
 It implies that $$X=C_{f}\simeq C_{h_{1}}\vee C_{f'}=\bigvee_{i}M_{\alpha_{i}}^{n+2}\vee C_{f'},~~ f'\in El(\Gamma). $$
Since for any $f\in El(\Gamma)$, $C(f)=C_{f}$ is 2-torsion free,  by the above analysis, we complete the proof of  Lemma \ref{lemma5.4}.

   $$\xymatrix{
  \s{(\bigvee_{k_{1}}S^{n+2})\vee A(k',l)} \ar@{=}[d] \ar[rrr]^{h_{1}\vee h'} &  &  & \s{(\bigvee_{k_{1}}S^{n+2})\vee B(t',u,v,w)} \ar@{=}[d] \\
  \s{A(k,l)} \ar@{_{(}->}[d]_{j_{A}} \ar[r]^{\alpha~\simeq} &\s{A(k,l)} \ar@{_{(}->}[d]_{j_{A}} \ar[r]^{h} & \s{B(t,u,v,w)} \ar[r]^{\beta~\simeq} &\s{B(t,u,v,w)}  \\
  \s{A} \ar@{=}[d] \ar[r]^{\alpha\vee 1_{\s{M_{A}}}~~\simeq} & \s{A} \ar[r]^{f} & \s{B} \ar@{->>}[u]_{p_{B}} \ar[r]^{\beta\vee 1_{\s{M_{B}}}~~\simeq} & \s{B} \ar@{=}[d] \ar@{->>}[u]_{p_{B}} \\
   \s{(\bigvee_{k_{1}}S^{n+2})\vee (A(k',l)\vee M_{A})} \ar@{.>}[rrr]^{h_{1}\vee f'} & & &\s{(\bigvee_{k_{1}}S^{n+2})\vee (B(t',u,v,w)\vee M_{B}) } }$$
      $$Diagram ~1$$
\end{proof}

Take $H_{0}=\Gamma$ in Corollary \ref{corollary3.3}, then $f_{1}af_{2}=0$ whenever $f_{i}\in Hos(B_{i},A_{i})$ $(i=1,2)$, $a\in \Gamma(A_{2},B_{1})$, $A_{i}\in\mathcal{A}$ and $B_{i}\in \mathcal{B}$.
 Hence by Corollary \ref{corollary3.3}, we have

 $$C: El(\Gamma)\xlongrightarrow{P \simeq_{rep}}El(\Gamma)/\mathcal{J}_{\Gamma}\xlongrightarrow{\overline{C}~\simeq}\mathcal{A}\dag_{n+2} \mathcal{B}/\mathcal{I}_{\Gamma}\xlongleftarrow{P \simeq_{rep}}\mathcal{A}\dag_{n+2}\mathcal{B},$$ which implies the following

 \begin{corollary}\label{corollary 5.5}
  In Lemma \ref{lemma5.4},
  $$\{~C(f)~\text{is indecomposable}~|~f\in El(\Gamma)~\}=ind(\mathcal{A}\dag_{n+2}\mathcal{B})\cong ind El(\Gamma).$$
 \end{corollary}

 In the remainder of this section, we will find the matrix problem corresponding to $El(\Gamma)$.

Computing $\Gamma(A,B)$ for $A\in ind\mathcal{A}$, $B\in ind\mathcal{B}$; $Hos(A,A')$ for $A, A'\in ind\mathcal{A}$ and $Hos(B,B')$ for $B, B'\in ind\mathcal{B}$ as in \cite{RefDrFP}. For example,

\begin{itemize}
  \item [] $\Gamma(S^{n+2}, S^{n+2})=\Gamma(S^{n+2},C_{\eta})=0$;
  \item [] $Hos(\mo{r}{n+2} , \mo{s}{n})=\begin{footnotesize}\begin{tabular}{c|cc|}
   \multicolumn{1}{c}{}&\multicolumn{2}{c}{$\mo{r}{n+2}:\s{ n+2~~n+3}$  }\\
       \cline{2-3}
      $\mo{s}{n}:\s{n}$ & ~~~~~~0 &$\mathbb{Z}/3$\\
      \quad ${\scriptstyle{n+1}}$ & ~~~~~~0& 0 \\
        \cline{2-3}
     \end{tabular} \end{footnotesize}~;$

  \item [] $Hos(\mo{r}{n+2} , C_{\eta})=\begin{footnotesize}\begin{tabular}{c|cc|}
   \multicolumn{1}{c}{}&\multicolumn{2}{c} {$\mo{r}{n+2}:\s{ n+2~~n+3}$  }\\
       \cline{2-3}
      $C_{\eta}:\s{n}$ & ~~~~~0 &$\mathbb{Z}/3$\\
      \quad ${\scriptstyle{n+1}}$ & ~~~~~0& 0 \\
        \cline{2-3}
     \end{tabular}\end{footnotesize}~; $

  \item [] $Hos(\mo{r}{n+2} ,\mo{s}{n+2})=\begin{footnotesize}\begin{tabular}{c|cc|}
   \multicolumn{1}{c}{}&\multicolumn{2}{c} {$ \mo{r}{n+2}:\s{ n+2~~n+3}$  }\\
       \cline{2-3}
      $\mo{s}{n+2}:\s{n+2}$ & $~~~\mathbb{Z}/3^{s=}$ &0\\
      \qquad ${\scriptstyle{n+3}}$ &~~~0& $\mathbb{Z}/3^{r=}$ \\
        \cline{2-3}
     \end{tabular}\end{footnotesize}~, $
\end{itemize}

 $$\text{where}~~\begin{footnotesize}\begin{tabular}{|cc|}
       \hline
      $\mathbb{Z}/3^{s=}$ &0\\
      0& $\mathbb{Z}/3^{r=}$ \\
        \hline
     \end{tabular}\end{footnotesize} = \left\{ \left. \left( \begin{array}{cc}
                                  \bar{a }& 0 \\
                                  0 & \bar{b} \\
                                \end{array}
                              \right) \right |\bar{a }\in \mathbb{Z}/3^{s} , \bar{b }\in \mathbb{Z}/3^{r} ~\text{and}~ 3^ra=3^sb ~\text{in} ~\mathbb{Z} .\right\}$$
          $$~=
           \left\{
            \begin{array}{ll}
             \left\{ \left. \left( \begin{array}{cc}
              \bar{a }& 0 \\
               0 & \s{\overline{3^{r-s}a }}\\
                \end{array}
                \right) \right |\bar{a }\in \mathbb{Z}/3^{s}, \overline{3^{r-s}a }\in \mathbb{Z}/3^{r} \right\} & \hbox{$r>s$} \\
                 \left\{ \left. \left( \begin{array}{cc}
                \s{\overline{3^{s-r}b}} & 0 \\
                   0 & \bar{b} \\
                    \end{array}
                    \right) \right |\overline{3^{s-r}b}\in \mathbb{Z}/3^{s} , \bar{b }\in \mathbb{Z}/3^{r}\right\} & \hbox{$r\leq s$.}
                     \end{array}
               \right.$$

 Now we get the matrix problem $(\widetilde{\mathscr{A}}, \widetilde{\mathcal{G}})$ corresponding to $El(\Gamma)$ as follows.
 The objects of $El(\Gamma)$ can be represented by block matrices $\gamma=(\gamma_{ij})$  with finite order in  $Table~ 4$  which provides the matrix set $\widetilde{\mathscr{A}}$, where block $\gamma_{ij}$ has entries from the $(ij)$-th cell of $Table~ 4$.  Morphisms $\gamma\rightarrow \gamma'$ are given by block matrices $\alpha=(\alpha_{ij})$ and $\beta=(\beta_{ij})$ from $Table~5$ and $Table~6$  respectively with proper order, which provide the admissible transformations $\widetilde{\mathcal{G}}$.

$$\begin{tabular}{|c|c|c|cc|cc|cc|c}
   \multicolumn{10}{c}{$\Gamma(\mathcal{A},\mathcal{B})$}\\
                    \multicolumn{10}{c}{}\\
 \hline
   \diagbox{$\mathcal{B}$}{$\mathcal{A}$} & $S^{n+2}$ & $S^{n+3}$& \multicolumn{2}{c|}{$\mo{}{n+2}:\s{n+2~~n+3}$} &  \multicolumn{2}{c|}{$\mo{2}{n+2}:\s{n+2~~n+3}$} &  \multicolumn{2}{c|}{$\mo{3}{n+2}:\s{n+2~~n+3}$} & $\cdots$ \\
  \hline
 $S^{n}$ &$\mathbb{Z}/2$&$\mathbb{Z}/24$&~~~0&$\mathbb{Z}/3$&~~~0&$\mathbb{Z}/3$&~~~0&$\mathbb{Z}/3$&$\cdots$\\
 \hline
  $S^{n+1}$ &$\mathbb{Z}/2$&$\mathbb{Z}/2$&~~~0&0&~~~0&0&~~~0&0&$\cdots$\\
 \hline
  $S^{n+2}$ &0&$\mathbb{Z}/2$&~~~0&0&~~~0&0&~~~0&0&$\cdots$\\
 \hline
  $C_{\eta}:\s{n}$ &0 &$\mathbb{Z}/12$&~~~0&$\mathbb{Z}/3$&~~~0&$\mathbb{Z}/3$&~~~0&$\mathbb{Z}/3$&$\cdots$\\
  ~~$\s{n+2}$ &0&0&~~~0&0&~~~0&0&~~~0&0&$\cdots$\\
 \hline
  $\mo{}{n}:\s{n}$ &0 &$\mathbb{Z}/3$&~~~0&$\mathbb{Z}/3$&~~~0&$\mathbb{Z}/3$&~~~0&$\mathbb{Z}/3$&$\cdots$\\
  ~~$\s{n+1}$ &0&0&~~~0&0&~~~0&0&~~~0&0&$\cdots$\\
 \hline
  $\mo{2}{n}:\s{n}$ &0 &$\mathbb{Z}/3$&~~~0&$\mathbb{Z}/3$&~~~0&$\mathbb{Z}/3$&~~~0&$\mathbb{Z}/3$&$\cdots$\\
  ~~$\s{n+1}$ &0&0&~~~0&0&~~~0&0&~~~0&0&$\cdots$\\
 \hline
 $\mo{3}{n}:\s{n}$ &0 &$\mathbb{Z}/3$&~~~0&$\mathbb{Z}/3$&~~~0&$\mathbb{Z}/3$&~~~0&$\mathbb{Z}/3$&$\cdots$\\
  ~~$\s{n+1}$ &0&0&~~~0&0&~~~0&0&~~~0&0&$\cdots$\\
 \hline
 $\cdots$& $\cdots$& $\cdots$& $\cdots$& $\cdots$& $\cdots$& $\cdots$& $\cdots$& $\cdots$& $\cdots$\\
 \multicolumn{10}{c}{}\\
 \multicolumn{10}{c}{$Table~4$}\\
 \end{tabular}
$$

$$Hos(\mathcal{A},\mathcal{A})$$
\begin{footnotesize}$$\begin{tabular}{|c|c|c|cc|cc|c|cc|c}
    \hline
   \diagbox{$\mathcal{A}$}{$\mathcal{A}$} & $S^{n+2}$ & $S^{n+3}$& \multicolumn{2}{c|}{$\mo{}{n+2}:\s{n+2~~n+3}$} &  \multicolumn{2}{c|}{$\mo{2}{n+2}:\s{n+2~~n+3}$} & $\cdots$& \multicolumn{2}{c|}{$\mo{r}{n+2}:\s{n+2~~n+3}$} & $\cdots$ \\
        \hline
   $S^{n+2}$  &$\mathbb{Z}$&$\mathbb{Z}/2$&0&0&0&0&$\cdots$&0&0&$\cdots$  \\
        \hline
  $S^{n+3}$  &0&$\mathbb{Z}$&0&$\mathbb{Z}/3$&0&$\mathbb{Z}/3^2$&$\cdots$&0&$\mathbb{Z}/3^r$&$\cdots$\\
        \hline
   $\mo{}{n+2}:\s{n+2}$&$\mathbb{Z}/3$&0&$\mathbb{Z}/3^=$&0&$\mathbb{Z}/3^=$&0&$\cdots$&$\mathbb{Z}/3^=$&0&$\cdots$\\
    $~~~\s{n+3}$&0&0&0&$\mathbb{Z}/3^=$&0&$\mathbb{Z}/3^{2=}$&$\cdots$&0&$\mathbb{Z}/3^{r=}$&$\cdots$\\
        \hline
   $\mo{2}{n+2}:\s{n+2}$&$\mathbb{Z}/3^{2}$&0&$\mathbb{Z}/3^{2=}$&0&$\mathbb{Z}/3^{2=}$&0&$\cdots$&$\mathbb{Z}/3^{2=}$&0&$\cdots$\\
    $~~~\s{n+3}$&0&0&0&$\mathbb{Z}/3^=$&0&$\mathbb{Z}/3^{2=}$&$\cdots$&0&$\mathbb{Z}/3^{r=}$&$\cdots$\\
        \hline
   $\cdots$ & $\cdots$ & $\cdots$ & $\cdots$ & $\cdots$ & $\cdots$ & $\cdots$ & $\cdots$ & $\cdots$ & $\cdots$ & $\cdots$\\
   \hline
     $\mo{r}{n+2}:\s{n+2}$&$\mathbb{Z}/3^{r}$&0&$\mathbb{Z}/3^{r=}$&0&$\mathbb{Z}/3^{r=}$&0&$\cdots$&$\mathbb{Z}/3^{r=}$&0&$\cdots$\\
    $~~~\s{n+3}$&0&0&0&$\mathbb{Z}/3^=$&0&$\mathbb{Z}/3^{2=}$&$\cdots$&0&$\mathbb{Z}/3^{r=}$&$\cdots$\\
        \hline
   $\cdots$ & $\cdots$ & $\cdots$ & $\cdots$ & $\cdots$ & $\cdots$ & $\cdots$ & $\cdots$ & $\cdots$ & $\cdots$ & $\cdots$\\
  \end{tabular}
$$\end{footnotesize}
$$Table~5$$

$$ \begin{tabular}{|c|c|c|c|cc|cc|cc|c|cc|c}
                   \multicolumn{14}{c}{$Hos(\mathcal{B},\mathcal{B})$}\\
                    \multicolumn{14}{c}{}\\
                    \hline
             \diagbox{$\mathcal{B}$}{$\mathcal{B}$} & $\s{S^{n}}$&$\s{S^{n+1}}$&$\s{S^{n+2}}$& \multicolumn{2}{c|}{$\s{C_{\eta}:n~~n+2}$}&\multicolumn{2}{c|}{$\s{\mo{}{n}:n~~n+1}$}&\multicolumn{2}{c|}{$\s{\mo{2}{n}:n~~n+1}$}
              &$\s\cdots$&\multicolumn{2}{c|}{$\s{\mo{r}{n}:n~~n+1}$}&$\s\cdots$\\
                  \hline
 $\s{S^{n}}$ &$\s{\mathbb{Z}}$&$\s{\mathbb{Z}/2}$&$\s{\mathbb{Z}/2}$&$\s{2\mathbb{Z}}$&0&0&0&0&0&$\s \cdots$&0&0&$\s\cdots$ \\
 \hline
  $\s{S^{n+1}}$ &0&$\s{\mathbb{Z}}$&$\s{\mathbb{Z}/2}$&0&0&0&0&0&0&$\s \cdots$&0&0&$\s\cdots$ \\
  \hline
  $\s{S^{n+2}}$ &0&0&$\s{\mathbb{Z}}$&0&$\s{\mathbb{Z}}$&0&0&0&0&$\s \cdots$&0&0&$\s\cdots$ \\
  \hline
 $\s{C_{\eta}:n}$ &$\s{\mathbb{Z}}$&0&$\s{\mathbb{Z}/2^{=}}$&$\s{\mathbb{Z}^=}$&0&0&0&0&0&$\s \cdots$&0&0&$\s\cdots$ \\
   ~~$\s{n+2}$ &0&0&$\s{2\mathbb{Z}^{=}}$&0&$\s{\mathbb{Z}^=}$&0&0&0&0&$\s \cdots$&0&0&$\s\cdots$  \\
     \hline
 $\s{\mo{}{n}:n}$ &$\s{\mathbb{Z}/3}$&0&0&$\s{\mathbb{Z}/3}$&0&$\s{\mathbb{Z}/3^=}$&0&$\s{\mathbb{Z}/3^=}$&0&$\s \cdots$&$\s{\mathbb{Z}/3^=}$&0&$\s\cdots$ \\
   ~~$\s{n+1}$ &0&0&0&0&0&0&$\s{\mathbb{Z}/3^=}$&0&$\s{\mathbb{Z}/3^{2=}}$&$\s \cdots$&0&$\s{\mathbb{Z}/3^{r=}}$&$\s\cdots$  \\
     \hline
 $\s{\mo{2}{n}:n}$ &$\s{\mathbb{Z}/3^2}$&0&0&$\s{\mathbb{Z}/3^2}$&0&$\s{\mathbb{Z}/3^{2=}}$&0&$\s{\mathbb{Z}/3^{2=}}$&0&$\s \cdots$&$\s{\mathbb{Z}/3^{2=}}$&0&$\s\cdots$ \\
   ~~$\s{n+1}$ &0&0&0&0&0&0&$\s{\mathbb{Z}/3^=}$&0&$\s{\mathbb{Z}/3^{2=}}$&$\s \cdots$&0&$\s{\mathbb{Z}/3^{r=}}$&$\s\cdots$  \\
     \hline
  $\s\cdots$ &$\s\cdots$&$\s\cdots$&$\s\cdots$&$\s\cdots$&$\s\cdots$&$\s\cdots$&$\s\cdots$&$\s\cdots$&$\s\cdots$&$\s \cdots$&$\s\cdots$&$\s\cdots$&$\s\cdots$ \\
  \hline
  $\s{\mo{r}{n}:n}$&$\s{\mathbb{Z}/3^r}$&0&0&$\s{\mathbb{Z}/3^r}$&0&$\s{\mathbb{Z}/3^{r=}}$&0&$\s{\mathbb{Z}/3^{r=}}$&0&$\s \cdots$&$\s{\mathbb{Z}/3^{r=}}$&0&$\s\cdots$ \\
   ~~$\s{n+1}$ &0&0&0&0&0&0&$\s{\mathbb{Z}/3^=}$&0&$\s{\mathbb{Z}/3^{2=}}$&$\s \cdots$&0&$\s{\mathbb{Z}/3^{r=}}$&$\s\cdots$  \\
     \hline
     $\s\cdots$ &$\s\cdots$&$\s\cdots$&$\s\cdots$&$\s\cdots$&$\s\cdots$&$\s\cdots$&$\s\cdots$&$\s\cdots$&$\s\cdots$&$\s \cdots$&$\s\cdots$&$\s\cdots$&$\s\cdots$ \\
       \multicolumn{14}{c}{}\\
       \multicolumn{14}{c}{$Table~6$}\\
 \end{tabular}$$

It is well known
$$ind \widetilde{\mathscr{A}}\cong ind El(\Gamma)$$

  We eliminate the zero stripes $\mo{r}{n+2}:\s{n+2}$ and $\mo{s}{n}:\s{n+1}$ of matrices in $\widetilde{\mathscr{A}}$ to simplify the matrix problem $(\widetilde{\mathscr{A}}, \widetilde{\mathcal{G}})$ to the following equivalent matrix problem $(\mathscr{A}', \mathcal{G}')$.

$$\begin{tabular}{|c|c|c|c|c|c|c}
   \multicolumn{7}{c}{$\Gamma'(\mathcal{A},\mathcal{B})$}\\
                    \multicolumn{7}{c}{}\\
 \hline
   \diagbox{$\mathcal{B}$}{$\mathcal{A}$} & $S^{n+2}$ & $S^{n+3}$& \multicolumn{1}{c|}{$\mo{}{n+2}$} &  \multicolumn{1}{c|}{$\mo{2}{n+2}$} &  \multicolumn{1}{c|}{$\mo{3}{n+2}$} & $\cdots$ \\
  \hline
 $S^{n}$ &$\mathbb{Z}/2$&$\mathbb{Z}/24$&$\mathbb{Z}/3$&$\mathbb{Z}/3$&$\mathbb{Z}/3$&$\cdots$\\
 \hline
  $S^{n+1}$ &$\mathbb{Z}/2$&$\mathbb{Z}/2$&0&0&0&$\cdots$\\
 \hline
  $S^{n+2}$ &0&$\mathbb{Z}/2$&0&0&0&$\cdots$\\
 \hline
  $C_{\eta}:\s{n}$ &0 &$\mathbb{Z}/12$&$\mathbb{Z}/3$&$\mathbb{Z}/3$&$\mathbb{Z}/3$&$\cdots$\\
  ~~$\s{n+2}$ &0&0&0&0&0&$\cdots$\\
 \hline
  $\mo{}{n}$ &0 &$\mathbb{Z}/3$&$\mathbb{Z}/3$&$\mathbb{Z}/3$&$\mathbb{Z}/3$&$\cdots$\\
 \hline
  $\mo{2}{n}$ &0 &$\mathbb{Z}/3$&$\mathbb{Z}/3$&$\mathbb{Z}/3$&$\mathbb{Z}/3$&$\cdots$\\
 \hline
 $\mo{3}{n}$ &0 &$\mathbb{Z}/3$&$\mathbb{Z}/3$&$\mathbb{Z}/3$&$\mathbb{Z}/3$&$\cdots$\\
 \hline
 $\cdots$&  $\cdots$&  $\cdots$& $\cdots$& $\cdots$& $\cdots$& $\cdots$\\
 \multicolumn{7}{c}{}\\
 \multicolumn{7}{c}{$Table~7$}\\
 \end{tabular}
$$

$$\begin{tabular}{|c|c|c|c|c|c|c|c}
      \multicolumn{8}{c}{$Hos'(\mathcal{A},\mathcal{A})$}\\
                    \multicolumn{8}{c}{}\\
    \hline
    \diagbox{$\mathcal{A}$}{$\mathcal{A}$}& $S^{n+2}$ & $S^{n+3}$& \multicolumn{1}{c|}{$\mo{}{n+2}$} &  \multicolumn{1}{c|}{$\mo{2}{n+2}$} & $\cdots$& \multicolumn{1}{c|}{$\mo{r}{n+2}$} & $\cdots$ \\
        \hline
   $S^{n+2}$  &$\mathbb{Z}$&$\mathbb{Z}/2$&0&0&$\cdots$&0&$\cdots$  \\
        \hline
  $S^{n+3}$  &0&$\mathbb{Z}$&$\mathbb{Z}/3$&$\mathbb{Z}/3^2$&$\cdots$&$\mathbb{Z}/3^r$&$\cdots$\\
        \hline
    $\mo{}{n+2}$&0&0&$\mathbb{Z}/3$&0&$\cdots$&0&$\cdots$\\
        \hline
    $\mo{2}{n+2}$&0&0&$\mathbb{Z}/3$&$\mathbb{Z}/3^{2}$&$\cdots$&0&$\cdots$\\
        \hline
   $\cdots$ & $\cdots$ & $\cdots$ & $\cdots$ & $\cdots$ & $\cdots$ & $\cdots$ & $\cdots$\\
   \hline
    $\mo{r}{n+2}$&0&0&$\mathbb{Z}/3$&$\mathbb{Z}/3^{2}$&$\cdots$&$\mathbb{Z}/3^{r}$&$\cdots$\\
        \hline
   $\cdots$ &  $\cdots$ & $\cdots$ & $\cdots$ & $\cdots$ & $\cdots$ & $\cdots$ & $\cdots$\\
  \multicolumn{8}{c}{}\\
 \multicolumn{8}{c}{$Table~8$}\\
\end{tabular}
$$

$$ \begin{tabular}{|c|c|c|c|cc|c|c|c|c|c}
                   \multicolumn{11}{c}{$Hos'(\mathcal{B},\mathcal{B})$}\\
                    \multicolumn{11}{c}{}\\
                    \hline
              \diagbox{$\mathcal{B}$}{$\mathcal{B}$}& $S^{n}$&$S^{n+1}$&$S^{n+2}$& \multicolumn{2}{c|}{$C_{\eta}:\s{n~~n+2}$}&\multicolumn{1}{c|}{$\mo{}{n}$}&\multicolumn{1}{c|}{$\mo{2}{n}$}
              &$\s\cdots$&\multicolumn{1}{c|}{$\mo{r}{n}$}&$\s\cdots$\\
                  \hline
 $S^{n}$ &$\mathbb{Z}$&$\mathbb{Z}/2$&$\mathbb{Z}/2$&$2\mathbb{Z}$&0&0&0&$\s \cdots$&0&$\s\cdots$ \\
 \hline
  $S^{n+1}$ &0&$\mathbb{Z}$&$\mathbb{Z}/2$&0&0&0&0&$\s \cdots$&0&$\s\cdots$ \\
  \hline
  $S^{n+2}$ &0&0&$\mathbb{Z}$&0&$\mathbb{Z}$&0&0&$\s \cdots$&0&$\s\cdots$ \\
  \hline
 $C_{\eta}:\s{n}$ &$\mathbb{Z}$&0&$\mathbb{Z}/2^{=}$&$\mathbb{Z}^=$&0&0&0&$\s \cdots$&0&$\s\cdots$ \\
   ~~$\s{n+2}$ &0&0&$2\mathbb{Z}^{=}$&0&$\mathbb{Z}^=$&0&0&$\s \cdots$&0&$\s\cdots$  \\
     \hline
     $\mo{}{n}$ &$\mathbb{Z}/3$&0&0&$\mathbb{Z}/3$&0&$\mathbb{Z}/3$&$\mathbb{Z}/3$&$\s \cdots$&$\mathbb{Z}/3$&$\s\cdots$ \\
     \hline
 $\mo{2}{n}$ &$\mathbb{Z}/3^2$&0&0&$\mathbb{Z}/3^2$&0&0&$\mathbb{Z}/3^{2}$&$\s \cdots$&$\mathbb{Z}/3^{2}$&$\s\cdots$ \\
     \hline
  $\s\cdots$ &$\s\cdots$&$\s\cdots$&$\s\cdots$&$\s\cdots$&$\s\cdots$&$\s\cdots$&$\s \cdots$&$\s\cdots$&$\s\cdots$&$\s\cdots$ \\
  \hline
  $\mo{r}{n}$&$\mathbb{Z}/3^r$&0&0&$\mathbb{Z}/3^r$&0&0&0&$\s \cdots$&$\mathbb{Z}/3^{r}$&$\s\cdots$ \\
     \hline
 \multicolumn{11}{c}{}\\
       \multicolumn{11}{c}{$Table~9$}\\
 \end{tabular}.$$

 In the above tables, $\mo{r}{n+2}$ represents the $\mo{r}{n+2}:\s{n+3}$-vertical stripe and $\mo{s}{n}$ represents the $\mo{s}{n}:\s{n}$-horizontal stripe. $Table~7$ provides the matrix set $\mathscr{A}'$; $Table~8$ and $Table~9$ provide the (non-trivial) admissible transformations $\mathcal{G}'$ :
\begin{itemize}
   \item [$\mathcal{G}'_{el}$] :  ``elementary-row~(column)'' transformations of each horizontal (vertical) stripe .
   \item [$\mathcal{G}'_{co}$] :
    \begin{itemize}
     \item [(i)] $W^{S^{n+2}}<W^{S^{n+3}}$
      \item [(ii)] $W^{S^{n+3}}<W^{\mo{r}{n+2}}$ ; $W^{\mo{r+1}{n+2}}<W^{\mo{r}{n+2}}$ for any $r\in \mathbb{N}_+$
        \end{itemize}
   \item [$\mathcal{G}'_{ro}$]:
 \begin{itemize}
   \item [(i)] $W_{S^{n+2}}<W_{S^{n+1}}<W_{S^{n}}$
    \item [(ii)] $W_{S^{n}}< W_{C_{\eta}:n}< W_{\mo{s}{n}}$ ; $ W_{\mo{s+1}{n}}< W_{\mo{s}{n}}$ ; $2W_{C_{\eta}:n}<W_{S^{n}}$  for any $s\in \mathbb{N}_+$
    \item [(iii)] $6W_{S^{n+2}}<W_{C_{\eta:n}}$
    \end{itemize}
   \end{itemize}

 \begin{remark}\label{remark 5.6}
 \begin{itemize}
   \item []
   \item [(1)] $W_{x}<W_{y}$ means that adding $k$ times of a row of $W_{x}$ to a row of $W_{y}$  is admissible and $aW_{x}<W_{y}$ ($a\in\mathbb{N}_+$) means  adding $ak$ times of a row of $W_{x}$ to a row of $W_{y}$ is admissible where $k$ is an any nonzero integer.  $W^{x}<W^{y}$ has the similar meaning for corresponding vertical stripes.
   \item [(2)] Similarly, zero blocks in $Table~7$  should keep being zero after admissible transformations. Adding $1\in \mathbb{Z}/2$ to an element $a\in \mathbb{Z}/24$ gives $a+12$ in $\mathbb{Z}/24$.
   \item [(3)] Special rules of matrix product in $Table~3$ are also needed for matrix product in $Table~9$.
 \end{itemize}
 \end{remark}

\section{Computation of $ind\mathscr{A}'$ for matrix problem  $(\mathscr{A}', \mathcal{G}')$}
\label{sec:6}
 In this section we solve the matrix problem $(\mathscr{A}', \mathcal{G}')$ to get $ind\mathscr{A}'$, then we get the $ind \mathbf{F}_{n(2)}^{4}$ by $ind\mathscr{A}'$.

\subsection{$p$-primary component of matrix problem  $(\mathscr{A}', \mathcal{G}')$ $p=2,3$}
\label{subsec:6.1}

 Let $\Gamma'(\mathcal{A}, \mathcal{B})_{(2)}$ be the $2$-primary component of $\Gamma'(\mathcal{A}, \mathcal{B})$, that means we replace $\mathbb{Z}/24$ by $\mathbb{Z}/8$, $\mathbb{Z}/12$ by $\mathbb{Z}/4$ and $\mathbb{Z}/3$ by 0 in $Table~7$. Similarly,  $\Gamma'(\mathcal{A}, \mathcal{B})_{(3)}$ is the $3$-primary component of $\Gamma'(\mathcal{A}, \mathcal{B})$, that means we replace $\mathbb{Z}/24$ by $\mathbb{Z}/3$, $\mathbb{Z}/12$ by $\mathbb{Z}/3$ and  $\mathbb{Z}/2$ by 0 in $Table~7$.

 Then we get the following two matrix problems $(\mathscr{A}'_{(2)}, \mathcal{G}')$ and $(\mathscr{A}'_{(3)}, \mathcal{G}')$ with admissible transformations also provided by $Table~8$ and $Table~9$.

$$\begin{tabular}{|c|c|c|c|c|c}
   \multicolumn{6}{c}{$\Gamma'(\mathcal{A}, \mathcal{B})_{(2)}$}\\
                    \multicolumn{6}{c}{}\\
 \hline
  \diagbox{$\mathcal{B}$}{$\mathcal{A}$}  & $S^{n+2}$ & $S^{n+3}$& \multicolumn{1}{c|}{$\mo{}{n+2}$} &  \multicolumn{1}{c|}{$\mo{2}{n+2}$}   & $\cdots$ \\
  \hline
 $S^{n}$ &$\mathbb{Z}/2$&$\mathbb{Z}/8$&0&0&$\cdots$\\
 \hline
  $S^{n+1}$ &$\mathbb{Z}/2$&$\mathbb{Z}/2$&0&0&$\cdots$\\
 \hline
  $S^{n+2}$ &0&$\mathbb{Z}/2$&0&0&$\cdots$\\
 \hline
  $C_{\eta}:\s{n}$ &0 &$\mathbb{Z}/4$&0&0&$\cdots$\\
  ~~$\s{n+2}$ &0&0&0&0&$\cdots$\\
 \hline
  $\mo{}{n}$ &0 &0&0&0&$\cdots$\\
 \hline
  $\mo{2}{n}$ &0 &0&0&0&$\cdots$\\
 \hline
 $\cdots$&  $\cdots$&  $\cdots$& $\cdots$& $\cdots$& $\cdots$\\
 \multicolumn{6}{c}{}\\
 \multicolumn{6}{c}{$Table~10$}\\
 \end{tabular}
$$

The list of non-trivial admissible transformations on $\Gamma'(\mathcal{A}, \mathcal{B})_{(2)}$ is :
\begin{itemize}
  \item [$el$]: ``elementary-row~(column)'' transformations of each horizontal (vertical) stripe ;
  \item [$co$]: $W^{S^{n+2}}<W^{S^{n+3}}$ ;
  \item [$ro$] : $W_{S^{n}}< W_{C_{\eta}:n}$ ; $2W_{C_{\eta}:n}<W_{S^{n}}$ ; $2W_{S^{n+2}}< W_{C_{\eta}:n}$ .
\end{itemize}

\
\
$$\begin{tabular}{|c|c|c|c|c|c|c}
   \multicolumn{7}{c}{$\Gamma'(\mathcal{A}, \mathcal{B})_{(3)}$}\\
                    \multicolumn{7}{c}{}\\
 \hline
  \diagbox{$\mathcal{B}$}{$\mathcal{A}$}  & $S^{n+2}$ & $S^{n+3}$& \multicolumn{1}{c|}{$\mo{}{n+2}$} &  \multicolumn{1}{c|}{$\mo{2}{n+2}$} &  \multicolumn{1}{c|}{$\mo{3}{n+2}$} & $\cdots$ \\
  \hline
 $S^{n}$ &0&$\mathbb{Z}/3$&$\mathbb{Z}/3$&$\mathbb{Z}/3$&$\mathbb{Z}/3$&$\cdots$\\
 \hline
  $S^{n+1}$ &0&0&0&0&0&$\cdots$\\
 \hline
  $S^{n+2}$ &0&0&0&0&0&$\cdots$\\
 \hline
  $C_{\eta}:\s{n}$ &0 &$\mathbb{Z}/3$&$\mathbb{Z}/3$&$\mathbb{Z}/3$&$\mathbb{Z}/3$&$\cdots$\\
  ~~$\s{n+2}$ &0&0&0&0&0&$\cdots$\\
 \hline
  $\mo{}{n}$ &0 &$\mathbb{Z}/3$&$\mathbb{Z}/3$&$\mathbb{Z}/3$&$\mathbb{Z}/3$&$\cdots$\\
 \hline
  $\mo{2}{n}$ &0 &$\mathbb{Z}/3$&$\mathbb{Z}/3$&$\mathbb{Z}/3$&$\mathbb{Z}/3$&$\cdots$\\
 \hline
 $\mo{3}{n}$ &0 &$\mathbb{Z}/3$&$\mathbb{Z}/3$&$\mathbb{Z}/3$&$\mathbb{Z}/3$&$\cdots$\\
 \hline
 $\cdots$&  $\cdots$&  $\cdots$& $\cdots$& $\cdots$& $\cdots$& $\cdots$\\
 \multicolumn{7}{c}{}\\
 \multicolumn{7}{c}{$Table~11$}\\
 \end{tabular}
$$

The list of non-trivial admissible transformations on $\Gamma'(\mathcal{A}, \mathcal{B})_{(3)}$ is :

\begin{itemize}
   \item [$el$] :  ``elementary-row~(column)'' transformations of each horizontal (vertical) stripe.
   \item [$co$] :   $W^{S^{n+3}}<W^{\mo{r}{n+2}}$; $W^{\mo{r+1}{n+2}}<W^{\mo{r}{n+2}}$ for any $r\in \mathbb{N}_+$
   \item [$ro$]:
   $W_{S^{n}}< W_{C_{\eta}:n}< W_{\mo{s}{n}}$ ; $ W_{\mo{s+1}{n}}< W_{\mo{s}{n}}$ ; $2W_{C_{\eta}:n}<W_{S^{n}}$  for any $s\in \mathbb{N}_+$.
   \end{itemize}

   Let $\Delta_{(2)(3)}$ be the subset of $\Gamma'(\mathcal{A} ,\mathcal{B})_{(2)}\times \Gamma'(\mathcal{A}, \mathcal{B})_{(3)}$ which consists of $(M_2,M_3)$  such that for every $x\in ind\mathcal{A}$, $y\in ind\mathcal{B}$, the two blocks $W_{x}^{y}$  respectively in matrixes $M_2$ and $M_3$
    have the same order. Define the map
  $$\Gamma'(\mathcal{A},\mathcal{B})\xlongrightarrow{L=(L_2, L_3)}\Delta_{(2)(3)}\subset \Gamma'(\mathcal{A} ,\mathcal{B})_{(2)}\times \Gamma'(\mathcal{A}, \mathcal{B})_{(3)}.$$

$L$ is given by the following ring  isomorphisms
$$\mathbb{Z}/24 \xlongrightarrow{L_{24}} \mathbb{Z}/8\times \mathbb{Z}/3 ~~,~~1\mapsto (1,1) ;$$
$$\mathbb{Z}/12 \xlongrightarrow{L_{12}} \mathbb{Z}/4\times \mathbb{Z}/3 ~~,~~1\mapsto (1,1) .$$

The inverse map $T$ of $L$
$$\Delta_{(2)(3)}\xlongrightarrow{T}\Gamma'(\mathcal{A}, \mathcal{B})$$
is given by the following two ring isomorphisms
$$ \mathbb{Z}/8\times \mathbb{Z}/3 \xlongrightarrow{T_{8}} \mathbb{Z}/24~~,~~(a ,b )\mapsto 9a+16b ;$$
$$ \mathbb{Z}/4\times \mathbb{Z}/3  \xlongrightarrow{T_{4}} \mathbb{Z}/12~~,~~(a,b)\mapsto 9a+4b .$$
which are the inverse of $L_{24}$ and $L_{12}$ respectively.

It's easy to know that if $M\cong N$ in matrix problem $(\mathscr{A}', \mathcal{G}')$ then  $L_2(M)\cong L_2(N)$ and
$L_3(M)\cong L_3(N)$ in matrix problem $(\mathscr{A}'_{(2)}, \mathcal{G}')$ and $(\mathscr{A}'_{(3)}, \mathcal{G}')$ respectively. We don't know whether the inverse is true. However, in the following we will show that the inverse will  be true if we take some restrictions to the admissible transformations on $\mathscr{A}'_{(2)}$ and $\mathscr{A}'_{(3)}$.
\\

Some notations :
\begin{itemize}
  \item [(1)] Let $Hos'(\mathcal{A}, \mathcal{A})(y_1,y_2,\cdots,y_n)$ (resp. $Hos'(\mathcal{B}, \mathcal{B})(x_1,x_2,\cdots,x_m)$) be the set of all square matrices in the $Hos'(\mathcal{A}, \mathcal{A})$ (resp. $Hos'(\mathcal{B}, \mathcal{B})$) with only $y_1,y_2,\cdots,y_n$-stripes (resp. $x_1,x_2,\cdots,x_m$-stripes  ).

   Especially, we denote $\mathcal{V}= Hos'(\mathcal{B}, \mathcal{B})(S^{n},S^{n+1},S^{n+2},C_{\eta}:{\s n})$ (note that there is no $C_{\eta}:\s{n+2}$-stripe). And we call the sub-matrix which contains entries in $S^{n}$, $S^{n+1}$, $S^{n+2}$, $C_{\eta}:\s n$-stripes of $M\in Hos'(\mathcal{B}, \mathcal{B})$ ``$\mathcal{V}$-part of $M$''.

  \item [(2)] Let $I$ be the identity matrix and $E_{ij}$ be the matrix whose unique non-zero entry has index $(i,j)$ and equals 1. Then the (i+$a$j)-type of  elementary row (column) transformations corresponds to left  (right) multiplication by an elementary matrix $I+aE_{ij}$ ($I+aE_{ji}$), and ($-1$i)-type of  elementary row (column) transformations corresponds to left (right) multiplication by an elementary matrix $I-2E_{ii}$.
      Note that the (i,j)-type of  transformations can be obtained by composition of  (i+$a$j)-type  and ($-1$i)-type of elementary transformations.
\end{itemize}

Let $\he{A}^+$ be the subset consists of  invertible matrices $\alpha=\begin{tabular}{|c|c|}
 \hline
  $U_1$& $U_2$ \\
    \hline
      0 & $U_4$ \\
        \hline
   \end{tabular}$
in $Hos'(\mathcal{A}, \mathcal{A})$, where  $U_1\in Hos'(\mathcal{A}, \mathcal{A})(S^{n+2},S^{n+3})$ is a product of elementary matrices $I+aE_{ij}$, $i\neq j$.

 Let $\he{B}^+$ be the subset consists of invertible matrices $\beta=\begin{tabular}{|c|c|}
  \hline
    $V_1$& 0 \\
      \hline
         $V_3$ & $V_4$ \\
         \hline
        \end{tabular} $
in   $Hos'(\mathcal{B},\mathcal{B})$, where
 $$V_1=\begin{tabular}{|c|c|c|cc|}
 \hline
  $V_{11}$ & $V_{12}$& $V_{13}$ & $2V_{14}$ &0 \\
  \hline
   0 & $V_{22}$ & $V_{23}$ & 0& 0 \\
   \hline
    0 &0 & $V_{33}$ & 0 & 0 \\
    \hline
   $V_{41}$ & 0 &$V_{43}$ & $V_{44}$ & 0 \\
    0 & 0 & 0 & 0 & $V_{55}$ \\
    \hline
     \end{tabular}
$$ which is an element in

$$ Hos'(\mathcal{B}, \mathcal{B})(S^{n},S^{n+1}, S^{n+2}, C_{\eta}: {\s{n}} , C_{\eta}:{\s{n+2}})=\begin{tabular}{|c|c|c|cc|}
 \hline
  $\mathbb{Z}$ & $\mathbb{Z}/2$ & $\mathbb{Z}/2$  & 2$\mathbb{Z}$  &0 \\
  \hline
   0 & $\mathbb{Z}$ & $\mathbb{Z}/2$ & 0& 0 \\
   \hline
    0 &0 & $\mathbb{Z}$ & 0 & $\mathbb{Z}$ \\
    \hline
   $\mathbb{Z}$ & 0 & $\mathbb{Z}/2^{=}$ & $\mathbb{Z}^=$ & 0 \\
    0 & 0 & $2\mathbb{Z}^{=}$ & 0 & $\mathbb{Z}^=$ \\
    \hline
     \end{tabular}$$
 such that the $\mathcal{V}$-part
$W_1=\begin{tabular}{|c|c|c|c|}
 \hline
  $V_{11}$ & $V_{12}$& $V_{13}$ & $2V_{14}$ \\
  \hline
   0 & $V_{22}$ & $V_{23}$ & 0 \\
   \hline
    0 &0 & $V_{33}$ & 0  \\
    \hline
   $V_{41}$ & 0 & $V_{43}$& $V_{44}$  \\
    \hline
     \end{tabular}$  of $V_1$ is a product of elementary matrices $I+aE_{ij}$, $i\neq j$.

 Denoting by $\mathcal{G}'^+$  the admissible transformations provided by $\he{A}^+$ and $\he{B}^+$ on $\Gamma'(\mathcal{A},\mathcal{B})_{(2)}$ and $\Gamma'(\mathcal{A},\mathcal{B})_{(3)}$,  we get two new matrix problems $(\mathscr{A}'_{(2)}, \mathcal{G}'^+)$  and $(\mathscr{A}'_{(3)}, \mathcal{G}'^+)$.
\\
\\
\textbf{The differences between $(\mathscr{A}'_{(2)}, \mathcal{G}'^+)$ and $(\mathscr{A}'_{(2)}, \mathcal{G}')$ }:
The list of non-trivial admissible transformations of matrix problem $(\mathscr{A}'_{(2)}, \mathcal{G}'^+)$  is the same as that of matrix problem $(\mathscr{A}'_{(2)}, \mathcal{G}')$  except that (-1i)-type of elementary transformations are not allowed, and (i,j)-type should be replaced by (i, -j)-type or (-i,j)-type, which means when we transport two rows (columns) of a  stripe, one row (column) $\alpha$ of them is replaced by $-\alpha$.
\\
\\
 \textbf{The differences between $(\mathscr{A}'_{(3)}, \mathcal{G}'^+)$ and $(\mathscr{A}'_{(3)}, \mathcal{G}')$ }:
 The list of non-trivial admissible transformations of matrix problem $(\mathscr{A}'_{(3)}, \mathcal{G}'^+)$  is the same as that of matrix problem $(\mathscr{A}'_{(3)}, \mathcal{G}')$   except that (-1i)-type of elementary transformations on the $W^{S^{n+3}}$, $W_{S^{n}}$ and $W_{C_{\eta}:{\s{n}}}$ are not allowed; (i,j)-type elementary transformations on $W^{S^{n+3}}$, $W_{S^{n}}$ and $W_{C_{\eta}:{\s{n}}}$ should be replaced by (i,-j)-type or (-i,j)-type.

\begin{theorem}\label{theorem6.1}
If $M_{(2)}\cong N_{(2)}$ in matrix problem $(\mathscr{A}'_{(2)}, \mathcal{G}'^+)$ and $M_{(3)}\cong N_{(3)}$ in matrix problem $(\mathscr{A}'_{(3)}, \mathcal{G}'^+)$ , then $T(M_{(2)} , M_{(3)})\cong T( N_{(2)}, N_{(3)})$  in matrix problem  $(\mathscr{A}', \mathcal{G}')$.
\end{theorem}

\begin{proof}
By the condition of the Theorem, we get that $\beta_2M_{(2)}\alpha_2=N_{(2)}$,  $\beta_3M_{(3)}\alpha_3=N_{(3)}$
where $\alpha_{2}, \alpha_3\in \he{A}^+$ and $\beta_2, \beta_3 \in \he{B}^+$.
Let $$\alpha_{2}=\begin{tabular}{|c|c|}
\hline
$U_1$& $U_2$ \\
\hline
0 & $U_4$ \\
 \hline
\end{tabular} ~~,
 ~~\alpha_{3}=\begin{tabular}{|c|c|}
\hline
$U'_1$& $U'_2$ \\
\hline
0 & $U'_4$ \\
 \hline
\end{tabular}~~,$$
where $U_1 , U'_1 \in Hos'(\mathcal{A},\mathcal{A})(S^{n+2},S^{n+3})$ .
  $$\beta_2=\begin{tabular}{|c|c|}
   \hline
   $V_1$& 0 \\
   \hline
    $V_3$ & $V_4$ \\
     \hline
     \end{tabular}~~,~~\beta_3=\begin{tabular}{|c|c|}
   \hline
   $V'_1$& 0 \\
   \hline
    $V'_3$ & $V'_4$ \\
     \hline
     \end{tabular}$$
where $V_1 , V'_1 \in Hos'(\mathcal{B},\mathcal{B})(S^{n},S^{n+1},S^{n+2},C_{\eta}:{\s n},C_{\eta}:{\s n+2})$.
 $\mathcal{V}$-part of $V_{1}$ and $V'_{1}$ are denoted by $W_{1}$ and $W'_{1}$ respectively.

\begin{lemma}\label{lemma6.2}
For any $$U_1 , U'_1 \in Hos'(\mathcal{A},\mathcal{A})(S^{n+2},S^{n+3})=
\begin{tabular}{|c|c|}
\hline
$\mathbb{Z}$ &  $\mathbb{Z}/2$ \\
\hline
0 &  $\mathbb{Z}$ \\
 \hline
 \end{tabular}$$
  and
$$W_1, W'_1\in \mathcal{V}=\begin{tabular}{|c|c|c|c|}
\hline
$\mathbb{Z}$ & $\mathbb{Z}/2$ & $\mathbb{Z}/2$ & $2\mathbb{Z}$ \\
\hline
0 & $\mathbb{Z}$ & $\mathbb{Z}/2$ & 0 \\
 \hline
  0& 0 & $\mathbb{Z}$ & 0 \\
  \hline
  $\mathbb{Z}$ & 0 & $\mathbb{Z}/2$ & $\mathbb{Z}$ \\
  \hline
  \end{tabular}$$
where $U_1$, $U'_1$, $W_1$ and $W'_1$ are products of elementary matrices $I+aE_{ij}~(i\neq j,~ a\in\mathbb{Z})$, orders of $U_1$, $U'_1$ (respectively orders of $W_1$, $W'_1$) are the same,
there exist invertible matrices
$$ U\in Hos'(\mathcal{A},\mathcal{A})(S^{n+2},S^{n+3}) ~,~~~~~ W\in \mathcal{V} $$  such that
$$\left\{
  \begin{array}{l}
    U\equiv U_1 ~(mod~8)  \\
    U\equiv U'_1 ~(mod~3)
  \end{array}
\right.    ~~~\text{and}~~~\left\{
  \begin{array}{l}
    W\equiv W_1 ~(mod~8)  \\
    W\equiv W'_1 ~(mod~3)
  \end{array}
\right. .$$
 Note. For any abelian group $A$, $a,b\in A$, and positive integer $k$, $a\equiv b~(mod~k)$ means that the image of $a$ and $b$ are equal under the quotient homomorphism $A \rightarrow A/kA$.
 \end{lemma}

We give some remarks before the proof of this lemma.

 Using $W$ and $U$ in Lemma \ref{lemma6.2}, let

$$      V=\begin{tabular}{c|ccc|c|}
            \multicolumn{4}{c} {} &  \multicolumn{1}{c} {$C_{\eta}:\s{n+2}$} \\
           \cline{2-5}
          &\multicolumn{3}{c|} {~~~~$\Large{W}\rule[20pt]{0pt}{3pt} $~~~~} &  0\\
           &&&&\\
          \cline{2-5}
        $C_{\eta}:\s{n+2}$   & \multicolumn{3}{c|} {0} & \scriptsize{$V_{55}$} \\
        \cline{2-5}
       \end{tabular}$$  where $V_{55}$ is  an invertible  matrix that makes $V$  be an element in
       \newline $Hos'(\mathcal{B}, \mathcal{B})(S^{n} , S^{n+1} , S^{n+2} , C_{\eta}:{\s{n}} , C_{\eta}:{\s{n+2}})$. And let
$$ \alpha=\begin{tabular}{|c|c|}
\hline
$U$& $U'_2$ \\
\hline
0 & $U'_4$ \\
 \hline
\end{tabular}  ~~~ \text{and}~~~\beta=\begin{tabular}{|c|c|}
   \hline
   $V$& 0 \\
   \hline
    $V'_3$ & $V'_4$ \\
     \hline
     \end{tabular} . $$
Then $\alpha$ and $\beta$  are invertible  and $\beta M_{(2)}\alpha=\beta_2 M_{(2)}\alpha_2=N_{(2)}$,  $\beta M_{(3)}\alpha=\beta_3 M_{(3)}\alpha_3=N_{(3)}$.
Since $\beta T(M_{(2)} , M_{(3)})\alpha =T(\beta M_{(2)}\alpha, \beta M_{(3)}\alpha)=T(N_{(2)}, N_{(3)} )$.
Thus :
$$T(M_{(2)}, M_{(3)})\cong T( N_{(2)}, N_{(3)}) ~ \text{in matrix problem}~ (\mathscr{A}', \mathcal{G}'). $$
\end{proof}

\begin{proof}[The proof of Lemma~\ref{lemma6.2}]

  \begin{statement}[1]~~For any $A, B\in SL_n(\mathbb{Z})$, there is a $C\in SL_n(\mathbb{Z})$, such that
   $$C\equiv A~(mod~8) ~~~~ \text{and}~~~ C\equiv B~(mod~3).$$
  \end{statement}
The  statement(1) follows from the following two conclusions in \cite{RefAlg} :
$$SL_n(\mathbb{Z})\xlongrightarrow {q} SL_n(\mathbb{Z}/24)~~\text{is surjective} ;$$
$$ SL_n(\mathbb{Z}/24)\xlongrightarrow {(q_1,q_2)}SL_n(\mathbb{Z}/8)\times SL_n(\mathbb{Z}/3)~~~\text{is isomorphic},$$
where $q , q_1 ,q_2$ are quotient maps.
 \\
 \\
 \begin{statement}[2]~~Suppose that
 $$
  A=I+aE_{ij}=\left(
    \begin{array}{ccccc}
      1 &  &  & &  \\
       &  &  & a_{ij} & \\
       &  & \ddots &  & \\
       &  &  &  & \\
       &  &  &  & 1 \\
    \end{array}
  \right), ~B=I+bE_{st}=\left(
    \begin{array}{ccccc}
      1 &  &  & &  \\
       &  &  & b_{st} & \\
       &  & \ddots &  & \\
       &  &  &  & \\
       &  &  &  & 1 \\
    \end{array}
  \right)~~\left(
             \begin{array}{c}
               i\neq j, a\in \mathbb{Z}\\
               s\neq t, b\in \mathbb{Z} \\
             \end{array}
           \right)
  $$
 are any two elementary matrices  in
 $\mathcal{V}$  (resp. $Hos'(\mathcal{A},\mathcal{A})(S^{n+2} ,S^{n+3})$ ) of the same order, then there is
 an invertible block matrix $C$ in $\mathcal{V}$ (resp. $Hos'(\mathcal{A},\mathcal{A})(S^{n+2} ,S^{n+3})$) such that  $C\equiv A ~(mod~8), C\equiv B~(mod~3)$.

 \end{statement}
\begin{proof}[The proof of statement(2)]
 We only prove the case for $A, B\in \mathcal{V}$  since the remaining case is much more easier.

  Note if $a_{ij}$ (resp. $b_{st}$) is from $\mathbb{Z}$ or $2\mathbb{Z}$ block, then $a_{ij}=a$ (resp. $b_{st}=b$); if $a_{ij}$ (resp. $b_{st}$) is from $\mathbb{Z}/2$ block, then $a_{ij}$ (resp. $b_{st}$) is the image of $a$ (resp. $b$) under the quotient map $\mathbb{Z}\rightarrow \mathbb{Z}/2$.
 \begin{itemize}
  \item If $b_{st}$ is from $\mathbb{Z}/2$ block, then $ b_{st}\equiv 0~(mod~3)$. For $a_{ij}$ from $\mathbb{Z}/2$ block, take $C=A$. For $a_{ij}$ from $ \mathbb{Z}$ or $2\mathbb{Z}$ block, there is a $c\in \mathbb{Z}$, such that $c\equiv a~ (mod~8)$ and $c\equiv 0 ~(mod~3)$. Take $C=I+cE_{ij}$.
  \item If $b_{st}$ is from $\mathbb{Z}$ or $2\mathbb{Z}$ block,
    \begin{itemize}
      \item [(i)]  $i\neq t$  or $j\neq s$
      \newline there is a  $d\in \mathbb{Z}$  such that  $d\equiv 0 ~ (mod~8)$ and $d\equiv b ~(mod~3)$.
         \newline ~~~~~If $a_{ij}$ is from $\mathbb{Z}/2$ block, take integer $c$ such that $c\equiv a~(mod~2)$;
         \newline ~~~~~If $a_{ij}$ is from $\mathbb{Z}$ or $2\mathbb{Z}$ block, take integer $c$ such that $c\equiv a~(mod~8)$ and $c\equiv 0 ~(mod~3)$.
        Then take $C=I+cE_{ij}+dE_{st}=(I+cE_{ij})(I+dE_{st})$  which is invertible in $\mathcal{V}$.
      \item [(ii)]  $i=t$ and $j=s$
      \newline In this case $a_{ij}$ must come from $\mathbb{Z}$ or $2\mathbb{Z}$ block.
       Suppose $i>j$. By statement (1), there is a matrix
  $$ X=\left(
   \begin{array}{cc}
     x_{11} &  x_{12} \\
      x_{21} & x_{22} \\
       \end{array}
        \right)\in SL_2(\mathbb{Z}) ~~~\text{ such that}$$
  $$X\equiv \left(
   \begin{array}{cc}
     1 & 0 \\
      a & 1 \\
       \end{array}
        \right) ~(mod~8)~~~\text{and}~~~X\equiv \left(
   \begin{array}{cc}
     1 & b \\
      0 & 1 \\
       \end{array}
        \right)~(mod~3).$$
     Take
     $$C=I-(1-x_{11})E_{jj}-(1-x_{22})E_{ii}+x_{12}E_{ji}+x_{21}E_{ij}.$$   Note that $x_{12}\in 2\mathbb{Z}$ and if $a\in 2\mathbb{Z}$ then $x_{21}\in 2\mathbb{Z}$, so $C$ is an element in $\mathcal{V}$.  It is easy to check that $C$ is invertible in
      $\mathcal{V}$ and $C\equiv A~(mod~8)$, $C\equiv B~(mod~3)$.
     The proof of the case $i<j$ is similar.
  \end{itemize}

\end{itemize}
\end{proof}
Now the proof of the Lemma~\ref{lemma6.2} is easily obtained  by statement (2).
\end{proof}
\

\subsection{The indecomposable isomorphic classes of $(\mathscr{A}'_{(2)}, \mathcal{G}'^+)$  and $(\mathscr{A}'_{(3)}, \mathcal{G}'^+)$ }
\label{subsec:6.2}

   ~~~~Note that the matrix problem $(\mathscr{A}'_{(2)}, \mathcal{G}')$ is essentially the same as the 2-primary component of the matrix problem  $(\mathscr{A}^{0}, \mathcal{G}^{0}).$  Thus we can get the list (denoted by \textbf{List(**)}) of the indecomposable isomorphic classes of $(\mathscr{A}'_{(2)}, \mathcal{G}')$ from the \textbf{List(*)}  by taking $v$ to its image of the quotient map $\mathbb{Z}/24\rightarrow \mathbb{Z}/8$ or  $\mathbb{Z}/12\rightarrow \mathbb{Z}/4$. It means \textbf{List(**)} are just the same as the \textbf{List(*)} except that the ranges of $v$ are different . That is
\begin{itemize}
  \item  $v\in \{1\}\subset \mathbb{Z}/4$ for the case (I);
  \item  $v\in \{1,2\}\subset \mathbb{Z}/8$ for the case (1),(2),(4),(5),(7) of (II) ;
  \item  $v\in \{1,2\}\subset \mathbb{Z}/4$ for the case (3),(6) of (II) ;
  \item  $v\in \{1,2,3,4\}\subset \mathbb{Z}/8$ for the case (III) .
\end{itemize}

From the \textbf{The differences between $(\mathscr{A}'_{(2)}, \mathcal{G}'^+)$ and $(\mathscr{A}'_{(2)}, \mathcal{G}')$}, we know that $M\in \Gamma'(\mathcal{A},\mathcal{B})_{(2)}$ is indecomposable in $(\mathscr{A}'_{(2)}, \mathcal{G}'^+)$ if and only if it is indecomposable in $(\mathscr{A}'_{(2)}, \mathcal{G}')$. But non-isomorphic matrices of $(\mathscr{A}'_{(2)}, \mathcal{G}'^+)$ may be isomorphic in $(\mathscr{A}'_{(2)}, \mathcal{G}')$.

For example,
$\begin{tabular}{c|c|}
   \multicolumn{1}{c}{}&\multicolumn{1}{c} {$\s{S^{n+3}}$}\\
       \cline{2-2}
     $\s{S^{n}}$ & $\s{1}$ \\
      \cline{2-2}
    \end{tabular}~~ $ and $ \begin{tabular}{c|c|}
   \multicolumn{1}{c}{}&\multicolumn{1}{c} {$\s{S^{n+3}}$}\\
       \cline{2-2}
     $\s{S^{n}}$ & $\s{-1}$ \\
      \cline{2-2}
    \end{tabular}~~$   (where 1, -1 $\in \mathbb{Z}/8$ ), which are isomorphic under $ \mathcal{G}'$, are not isomorphic under $ \mathcal{G}'^+$.
\\

 Here are the list of the indecomposable isomorphic classes of $(\mathscr{A}'_{(2)}, \mathcal{G}'^+)$ :

\begin{itemize}
  \item  [] \textbf{List(2)} :
  \item [(I)] $\begin{tabular}{c|c|}
   \multicolumn{1}{c}{}&\multicolumn{1}{c} {$\s{S^{n+3}}$}\\
       \cline{2-2}
     $\s{S^{n+2}}$ & 1 \\
       \cline{2-2}
      $\s{C_{\eta}:{n}}$ & 1 \\
      \quad ${\s{n+2}}$ & 0 \\
        \cline{2-2}
     \end{tabular}$ ;
  \item [(II)]
       $(1)~\begin{tabular}{c|c|c|}
   \multicolumn{1}{c}{}&\multicolumn{1}{c} {$\s{S^{n+2}}$}&\multicolumn{1}{c} {$\s{S^{n+3}}$}\\
       \cline{2-3}
     $\s{S^{n}}$ & 1& $v$ \\
        \cline{2-3}
      $\s{S^{n+1}}$ & 0& 1 \\
       \cline{2-3}
     \end{tabular}$ ~~;~~$(2)~\begin{tabular}{c|c|c|}
   \multicolumn{1}{c}{}&\multicolumn{1}{c} {$\s{S^{n+2}}$}&\multicolumn{1}{c} {$\s{S^{n+3}}$}\\
       \cline{2-3}
     $\s{S^{n}}$ & 1& $v$ \\
        \cline{2-3}
      $\s{S^{n+2}}$ & 0& 1 \\
       \cline{2-3}
     \end{tabular}$~~;~~$(4)~\begin{tabular}{c|c|c|}
      \multicolumn{1}{c}{}&\multicolumn{1}{c} {$\s{S^{n+2}}$}&\multicolumn{1}{c} {$\s{S^{n+3}}$}\\
       \cline{2-3}
     $\s{S^{n}}$ & 1& $v$ \\
      \cline{2-3}
     \end{tabular}$
     \newline

     $(5)~\begin{tabular}{c|c|}
   \multicolumn{1}{c}{}&\multicolumn{1}{c} {$\s{S^{n+3}}$}\\
       \cline{2-2}
     $\s{S^{n}}$ & $v$ \\
       \cline{2-2}
      $\s{S^{n+1}}$ & $1$\\
        \cline{2-2}
     \end{tabular} $~~;~~ $(7)~\begin{tabular}{c|c|}
   \multicolumn{1}{c}{}&\multicolumn{1}{c} {$\s{S^{n+3}}$}\\
       \cline{2-2}
     $\s{S^{n}}$ & $v$ \\
       \cline{2-2}
      $\s{S^{n+2}}$ & $1$\\
        \cline{2-2}
     \end{tabular} $ ,
     \\
     \\
     \\
  where $v\in\{1,2,3\}\subset \mathbb{Z}/8$ for the cases (1),(2),(4),(5),(7).

     $(3)~\begin{tabular}{c|c|}
   \multicolumn{1}{c}{}&\multicolumn{1}{c} {$\s{S^{n+3}}$}\\
       \cline{2-2}
     $\s{S^{n+2}}$ & 1 \\
       \cline{2-2}
      $\s{C_{\eta}:{n}}$ & $v$\\
      \quad ${\s{n+2}}$ & 0 \\
        \cline{2-2}
     \end{tabular}$~~;~~
      $(6)~\begin{tabular}{c|c|}
   \multicolumn{1}{c}{}&\multicolumn{1}{c} {$\s{S^{n+3}}$}\\
       \cline{2-2}
   $\s{C_{\eta}:{n}}$ & $v$\\
      \quad ${\s{n+2}}$ & 0 \\
        \cline{2-2}
     \end{tabular} $ ,
     \\ \\  \\
      where $v\in\{1,2,3\}\subset \mathbb{Z}/4$ for the cases (3) (6).

        \item [(III)]
         $\begin{tabular}{c|c|}
   \multicolumn{1}{c}{}&\multicolumn{1}{c} {$\s{S^{n+3}}$}\\
       \cline{2-2}
     $\s{S^{n}}$ & $v$ \\
      \cline{2-2}
    \end{tabular} $  ~~~~~~~~where $v\in\{1,2,\cdots,7\}\subset \mathbb{Z}/8$.

         \item [(IV)] $(1)\begin{tabular}{c|c|}
   \multicolumn{1}{c}{}&\multicolumn{1}{c} {$\s{S^{n+2}}$}\\
       \cline{2-2}
     $\s{S^{n+1}}$ & $1$ \\
      \cline{2-2}
    \end{tabular}$~~;~~$(2)~\begin{tabular}{c|c|}
   \multicolumn{1}{c}{}&\multicolumn{1}{c} {$\s{S^{n+3}}$}\\
       \cline{2-2}
     $\s{S^{n+2}}$ & $1$ \\
      \cline{2-2}
    \end{tabular}~~;$

    $(3)~\begin{tabular}{c|c|}
   \multicolumn{1}{c}{}&\multicolumn{1}{c} {$\s{S^{n+2}}$}\\
       \cline{2-2}
     $\s{S^{n}}$ & $1$ \\
      \cline{2-2}
    \end{tabular}~~; $~~$(4)~\begin{tabular}{c|c|}
   \multicolumn{1}{c}{}&\multicolumn{1}{c} {$\s{S^{n+3}}$}\\
       \cline{2-2}
     $\s{S^{n+1}}$ & $1$ \\
      \cline{2-2}
    \end{tabular}~~. $
   \end{itemize}

   For the matrix problem  $(\mathscr{A}'_{(3)}, \mathcal{G}'^+)$, the indecomposable isomorphic classes are given by

   \begin{itemize}
     \item []  \textbf{List(3)} :
     \item []   $\begin{tabular}{c|c|}
   \multicolumn{1}{c}{}&\multicolumn{1}{c} {$\s{S^{n+3}}$}\\
       \cline{2-2}
     $\s{S^{n}}$ &1\\
      \cline{2-2}
    \end{tabular} $~~;~~ $\begin{tabular}{c|c|}
   \multicolumn{1}{c}{}&\multicolumn{1}{c} {$\s{S^{n+3}}$}\\
       \cline{2-2}
     $\s{S^{n}}$ & -1 \\
      \cline{2-2}
    \end{tabular} $~~;~~$\begin{tabular}{c|c|}
   \multicolumn{1}{c}{}&\multicolumn{1}{c} {$\s{S^{n+3}}$}\\
       \cline{2-2}
      $\s{C_{\eta}:{n}}$ & 1\\
      \quad ${\s{n+2}}$ & 0 \\
        \cline{2-2}
     \end{tabular}$~~;~~$\begin{tabular}{c|c|}
   \multicolumn{1}{c}{}&\multicolumn{1}{c} {$\s{S^{n+3}}$}\\
       \cline{2-2}
      $\s{C_{\eta}:{n}}$ & -1\\
      \quad ${\s{n+2}}$ & 0 \\
        \cline{2-2}
     \end{tabular}$

     \item []
     $\begin{tabular}{c|c|}
   \multicolumn{1}{c}{}&\multicolumn{1}{c}{$\s{S^{n+3}}$}\\
       \cline{2-2}
     $\s{M_{3^{s}}^{n}}$ & 1 \\
      \cline{2-2}
    \end{tabular} $~~;~~$\begin{tabular}{c|c|}
   \multicolumn{1}{c}{}&\multicolumn{1}{c}{$\s{M_{3^{r}}^{n+2}}$}\\
       \cline{2-2}
     $\s{M_{3^{s}}^{n}}$ & 1 \\
      \cline{2-2}
    \end{tabular} $~~;~~$\begin{tabular}{c|c|}
   \multicolumn{1}{c}{}&\multicolumn{1}{c}{$\s{M_{3^{r}}^{n+2}}$}\\
       \cline{2-2}
     $\s{S^n}$ & 1 \\
      \cline{2-2}
    \end{tabular} $~~;~~$\begin{tabular}{c|c|}
   \multicolumn{1}{c}{}&\multicolumn{1}{c} {$\s{M_{3^{r}}^{n+2}}$}\\
       \cline{2-2}
      $\s{C_{\eta}:{n}}$ & 1\\
      \quad ${\s{n+2}}$ & 0 \\
        \cline{2-2}
     \end{tabular}$ .
    \\ \\ \\
     where $1$, $-1$ $\in \mathbb{Z}/3$  and $s , n \in \mathbb{N}_+$.
    \end{itemize}

  \subsection{The indecomposable isomorphic classes of $(\mathscr{A}', \mathcal{G}')$  and $\mathbf{F}_{n(2)}^{4} (n\geq5)$ }
\label{subsec:6.3}
By theorem~\ref{theorem6.1}, for any $M\in \Gamma'(\mathcal{A},\mathcal{B})$, we have $M\cong T(N_2 ,N_3)$ in matrix problem $(\mathscr{A}', \mathcal{G}')$ for some $N_2\in \Gamma'(\mathcal{A},\mathcal{B})_{(2)}$ and  $N_3\in \Gamma'(\mathcal{A},\mathcal{B})_{(3)}$, where
$$N_2=\bigoplus_{i}N_2^i\bigoplus \mathbf{O}_{2}, ~N_2^i~~  \text{is an indecomposable matrix listed in the \textbf{List(2)} for every}~i.$$
$$N_3=\bigoplus_{j}N_3^j\bigoplus \mathbf{O}_{3}, ~N_3^j~~  \text{is an indecomposable matrix listed in the \textbf{List(3)} for every}~j.$$
$\mathbf{O}_{2}$ and $\mathbf{O}_{3}$ are direct products of  some zero matrices.

Since $N_3$ is a matrix of which every row and every column have at most one nonzero entry, it enables us to select from $T(N_2 ,N_3)$ a set of indecomposable  matrices as follows which covers all the indecomposable isomorphic classes of matrix problem $(\mathscr{A}', \mathcal{G}')$.

\begin{itemize}
  \item [(I)]
    \begin{itemize}
    \item [] \begin{footnotesize}$\begin{tabular}{c|c|}
   \multicolumn{1}{c}{}&\multicolumn{1}{c} {$S^{n+3}$}\\
       \cline{2-2}
     $S^{n+2}$ & 1 \\
       \cline{2-2}
      $C_{\eta}:\s{{n}}$ & $T_4(a,b)$\\
      \quad ${\s{n+2}}$ & 0 \\
        \cline{2-2}
     \end{tabular}$\end{footnotesize} ~;
     ~$\begin{footnotesize} \begin{tabular}{c|c|c|}
   \multicolumn{1}{c}{}&\multicolumn{1}{c} {$S^{n+3}$}&\multicolumn{1}{c} {$\mo{r}{n+2}$}\\
       \cline{2-3}
     $S^{n+2}$ & 1 &0\\
       \cline{2-3}
      $C_{\eta}:\s{n}$ & $T_4(1,0)$&1\\
      \quad ${\s{n+2}}$ & 0 &0\\
        \cline{2-3}
     \end{tabular}\end{footnotesize}$ ~;

    \item []  $\begin{footnotesize}\begin{tabular}{c|c|}
   \multicolumn{1}{c}{}&\multicolumn{1}{c} {$S^{n+3}$}\\
       \cline{2-2}
     $S^{n+2}$ & 1 \\
       \cline{2-2}
      $C_{\eta}:\s{n}$ & $T_4(1,0)$\\
      \quad ${\s{n+2}}$ & 0 \\
        \cline{2-2}
       $\mo{s}{n}$ & 1\\
        \cline{2-2}
     \end{tabular}\end{footnotesize}$ ~;~$\begin{footnotesize}\begin{tabular}{c|c|c|}
   \multicolumn{1}{c}{}&\multicolumn{1}{c} {$S^{n+3}$}&\multicolumn{1}{c} {$\mo{r}{n+2}$}\\
       \cline{2-3}
     $S^{n+2}$ & 1 &0\\
       \cline{2-3}
      $C_{\eta}:\s{n}$ & $T_4(1,0)$&1\\
      \quad ${\s{n+2}}$ & 0 &0\\
        \cline{2-3}
         $\mo{s}{n}$ & 1& 0\\
        \cline{2-3}
     \end{tabular}\end{footnotesize}~~,$
    \\ \\ \\
  where $(a,b)\in \mathbb{Z}/4\times \mathbb{Z}/3$ such that $a\in \{0 ,1\}$ and $(a,b)\neq (0,0)$ . $s ,r \in \mathbb{N}_+$.

  \end{itemize}

     \item [(II)]
   \begin{itemize}
     \item [(1)] $\begin{footnotesize}\begin{tabular}{c|c|c|}
   \multicolumn{1}{c}{}&\multicolumn{1}{c} {$S^{n+2}$}&\multicolumn{1}{c} {$S^{n+3}$}\\
       \cline{2-3}
     $S^{n}$ & 1& $T_8(a,b)$ \\
        \cline{2-3}
      $S^{n+1}$ & 0& 1 \\
       \cline{2-3}
     \end{tabular}\end{footnotesize}$ ~~;~~   $\begin{footnotesize}\begin{tabular}{c|c|c|c|}
   \multicolumn{1}{c}{}&\multicolumn{1}{c} {$S^{n+2}$}&\multicolumn{1}{c} {$S^{n+3}$}&\multicolumn{1}{c} {$\mo{r}{n+2}$}\\
       \cline{2-4}
     $S^{n}$ & 1& $T_8(a,0)$ & 1\\
        \cline{2-4}
      $S^{n+1}$ & 0& 1 & 0\\
       \cline{2-4}
     \end{tabular}\end{footnotesize}$~~;

     \item [] $\begin{footnotesize}\begin{tabular}{c|c|c|}
   \multicolumn{1}{c}{}&\multicolumn{1}{c} {$S^{n+2}$}&\multicolumn{1}{c} {$S^{n+3}$}\\
       \cline{2-3}
     $S^{n}$ & 1& $T_8(a,0)$ \\
        \cline{2-3}
      $S^{n+1}$ & 0& 1 \\
       \cline{2-3}
        $\mo{s}{n}$ & 0& 1 \\
       \cline{2-3}
     \end{tabular}\end{footnotesize}$ ~~;~~   $\begin{footnotesize}\begin{tabular}{c|c|c|c|}
   \multicolumn{1}{c}{}&\multicolumn{1}{c} {$S^{n+2}$}&\multicolumn{1}{c} {$S^{n+3}$}&\multicolumn{1}{c} {$\mo{r}{n+2}$}\\
       \cline{2-4}
     $S^{n}$ & 1& $T_8(a,0)$ & 1\\
        \cline{2-4}
      $S^{n+1}$ & 0& 1 & 0\\
       \cline{2-4}
       $\mo{s}{n}$ & 0& 1&0 \\
      \cline{2-4}
     \end{tabular}\end{footnotesize}$~~;
    \end{itemize}

   \begin{itemize}
     \item [(2)] $\begin{footnotesize}\begin{tabular}{c|c|c|}
   \multicolumn{1}{c}{}&\multicolumn{1}{c} {$S^{n+2}$}&\multicolumn{1}{c} {$S^{n+3}$}\\
       \cline{2-3}
     $S^{n}$ & 1& $T_8(a,b)$ \\
        \cline{2-3}
      $S^{n+2}$ & 0& 1 \\
       \cline{2-3}
     \end{tabular}\end{footnotesize}$ ~~;~~   $\begin{footnotesize}\begin{tabular}{c|c|c|c|}
   \multicolumn{1}{c}{}&\multicolumn{1}{c} {$S^{n+2}$}&\multicolumn{1}{c} {$S^{n+3}$}&\multicolumn{1}{c} {$\mo{r}{n+2}$}\\
       \cline{2-4}
     $S^{n}$ & 1& $T_8(a,0)$ & 1\\
        \cline{2-4}
      $S^{n+2}$ & 0& 1 & 0\\
       \cline{2-4}
     \end{tabular}\end{footnotesize}$~~;

     \item [] $\begin{footnotesize}\begin{tabular}{c|c|c|}
   \multicolumn{1}{c}{}&\multicolumn{1}{c} {$S^{n+2}$}&\multicolumn{1}{c} {$S^{n+3}$}\\
       \cline{2-3}
     $S^{n}$ & 1& $T_8(a,0)$ \\
        \cline{2-3}
      $S^{n+2}$ & 0& 1 \\
       \cline{2-3}
        $\mo{s}{n}$ & 0& 1 \\
       \cline{2-3}
     \end{tabular}\end{footnotesize}$ ~~;~~   $\begin{footnotesize}\begin{tabular}{c|c|c|c|}
   \multicolumn{1}{c}{}&\multicolumn{1}{c} {$S^{n+2}$}&\multicolumn{1}{c} {$S^{n+3}$}&\multicolumn{1}{c} {$\mo{r}{n+2}$}\\
       \cline{2-4}
     $S^{n}$ & 1& $T_8(a,0)$ & 1\\
        \cline{2-4}
      $S^{n+2}$ & 0& 1 & 0\\
       \cline{2-4}
       $\mo{s}{n}$ & 0& 1&0 \\
      \cline{2-4}
     \end{tabular}\end{footnotesize}$~~;
    \end{itemize}

   \begin{itemize}
     \item [(3)] $\begin{footnotesize}\begin{tabular}{c|c|}
   \multicolumn{1}{c}{}&\multicolumn{1}{c} {$S^{n+3}$}\\
       \cline{2-2}
     $S^{n+1}$ & 1 \\
       \cline{2-2}
      $C_{\eta}:\s{n}$ & $T_4(a,b)$\\
      \quad ${\s{n+2}}$ & 0 \\
        \cline{2-2}
     \end{tabular}\end{footnotesize}$ ~~;~~$\begin{footnotesize}\begin{tabular}{c|c|c|}
   \multicolumn{1}{c}{}&\multicolumn{1}{c} {$S^{n+3}$}&\multicolumn{1}{c} {$\mo{r}{n+2}$}\\
       \cline{2-3}
     $S^{n+1}$ & 1 &0\\
       \cline{2-3}
      $C_{\eta}:\s{n}$ & $T_4(a,0)$&1\\
      \quad ${\s{n+2}}$ & 0 &0\\
        \cline{2-3}
     \end{tabular}\end{footnotesize}$~;

     \item [] $\begin{footnotesize}\begin{tabular}{c|c|}
   \multicolumn{1}{c}{}&\multicolumn{1}{c} {$S^{n+3}$}\\
       \cline{2-2}
     $S^{n+2}$ & 1 \\
       \cline{2-2}
      $C_{\eta}:\s{n}$ & $T_4(a,0)$\\
      \quad ${\s{n+2}}$ & 0 \\
        \cline{2-2}
       $\mo{s}{n}$ & 1\\
        \cline{2-2}
     \end{tabular}\end{footnotesize}$ ~~;~~$\begin{footnotesize}\begin{tabular}{c|c|c|}
   \multicolumn{1}{c}{}&\multicolumn{1}{c} {$S^{n+3}$}&\multicolumn{1}{c} {$\mo{r}{n+2}$}\\
       \cline{2-3}
     $S^{n+2}$ & 1 &0\\
       \cline{2-3}
      $C_{\eta}:\s{n}$ & $T_4(a,0)$&1\\
      \quad ${\s{n+2}}$ & 0 &0\\
        \cline{2-3}
         $\mo{s}{n}$ & 1& 0\\
        \cline{2-3}
     \end{tabular}\end{footnotesize}$~;
    \end{itemize}

     \begin{itemize}
       \item [(4)]$\begin{footnotesize}\begin{tabular}{c|c|c|}
      \multicolumn{1}{c}{}&\multicolumn{1}{c} {$S^{n+2}$}&\multicolumn{1}{c} {$S^{n+3}$}\\
       \cline{2-3}
     $S^{n}$ & 1& $T_8(a,b)$  \\
      \cline{2-3}
     \end{tabular}\end{footnotesize}$~~;~~$\begin{footnotesize}\begin{tabular}{c|c|c|c|}
      \multicolumn{1}{c}{}&\multicolumn{1}{c} {$S^{n+2}$}&\multicolumn{1}{c} {$S^{n+3}$}&\multicolumn{1}{c} {$\mo{r}{n+2}$}\\
       \cline{2-4}
     $S^{n}$ & 1& $T_8(a,0)$ &1 \\
      \cline{2-4}
     \end{tabular}\end{footnotesize}$~~;

       \item []$\begin{footnotesize}\begin{tabular}{c|c|c|}
      \multicolumn{1}{c}{}&\multicolumn{1}{c} {$S^{n+2}$}&\multicolumn{1}{c} {$S^{n+3}$}\\
       \cline{2-3}
     $S^{n}$ & 1& $T_8(a,0)$  \\
      \cline{2-3}
     $\mo{s}{n}$ & 0& 1 \\
      \cline{2-3}
     \end{tabular}\end{footnotesize}$~~;~~$\begin{footnotesize}\begin{tabular}{c|c|c|c|}
      \multicolumn{1}{c}{}&\multicolumn{1}{c} {$S^{n+2}$}&\multicolumn{1}{c} {$S^{n+3}$}&\multicolumn{1}{c} {$\mo{r}{n+2}$}\\
       \cline{2-4}
     $S^{n}$ & 1& $T_8(a,0)$ &1 \\
      \cline{2-4}
        $\mo{s}{n}$ & 0& 1 &0\\
          \cline{2-4}
     \end{tabular}\end{footnotesize}$~~;
     \end{itemize}

     \begin{itemize}
       \item [(5)]$\begin{footnotesize}\begin{tabular}{c|c|}
   \multicolumn{1}{c}{}&\multicolumn{1}{c} {$S^{n+3}$}\\
       \cline{2-2}
     $S^{n}$ &  $T_8(a,b)$  \\
       \cline{2-2}
      $S^{n+1}$ & $1$\\
        \cline{2-2}
     \end{tabular}\end{footnotesize} $~~;~~ $\begin{footnotesize}\begin{tabular}{c|c|c|}
   \multicolumn{1}{c}{}&\multicolumn{1}{c} {$S^{n+3}$}&\multicolumn{1}{c} {$\mo{r}{n+2}$}\\
       \cline{2-3}
     $S^{n}$ & $T_8(a,0)$  &1\\
       \cline{2-3}
      $S^{n+1}$ & $1$&0\\
        \cline{2-3}
     \end{tabular}\end{footnotesize} $~~;

       \item []$\begin{footnotesize}\begin{tabular}{c|c|}
   \multicolumn{1}{c}{}&\multicolumn{1}{c} {$S^{n+3}$}\\
       \cline{2-2}
     $S^{n}$ &  $T_8(a,b)$  \\
       \cline{2-2}
      $S^{n+1}$ & $1$\\
        \cline{2-2}
      $\mo{s}{n}$ & $1$\\
        \cline{2-2}
     \end{tabular}\end{footnotesize} $~~;~~ $\begin{footnotesize}\begin{tabular}{c|c|c|}
   \multicolumn{1}{c}{}&\multicolumn{1}{c} {$S^{n+3}$}&\multicolumn{1}{c} {$\mo{r}{n+2}$}\\
       \cline{2-3}
     $S^{n}$ & $T_8(a,0)$  &1\\
       \cline{2-3}
      $S^{n+1}$ & $1$&0\\
        \cline{2-3}
       $\mo{s}{n}$& $1$&0\\
        \cline{2-3}
     \end{tabular}\end{footnotesize} $~~;
     \end{itemize}

     \begin{itemize}
       \item [(6)] $ \begin{footnotesize}\begin{tabular}{c|c|}
   \multicolumn{1}{c}{}&\multicolumn{1}{c} {$S^{n+3}$}\\
       \cline{2-2}
   $C_{\eta}:\s{n}$ & $T_4(a,b)$   \\
      \quad ${\s{n+2}}$ & 0 \\
        \cline{2-2}
     \end{tabular}\end{footnotesize} $ ~~;~~$ \begin{footnotesize}\begin{tabular}{c|c|c|}
   \multicolumn{1}{c}{}&\multicolumn{1}{c} {$S^{n+3}$}&\multicolumn{1}{c} {$\mo{r}{n+2}$}\\
       \cline{2-3}
   $C_{\eta}:\s{n}$ & $T_4(a,0)$ &1  \\
      \quad ${\s{n+2}}$ & 0 &0\\
        \cline{2-3}
     \end{tabular}\end{footnotesize} $ ~~;

       \item [] $ \begin{footnotesize}\begin{tabular}{c|c|}
   \multicolumn{1}{c}{}&\multicolumn{1}{c} {$S^{n+3}$}\\
       \cline{2-2}
   $C_{\eta}:\s{n}$ & $T_4(a,0)$   \\
      \quad ${\s{n+2}}$ & 0 \\
        \cline{2-2}
$\mo{s}{n}$ & 1 \\
        \cline{2-2}
     \end{tabular}\end{footnotesize} $ ~~;~~$ \begin{footnotesize}\begin{tabular}{c|c|c|}
   \multicolumn{1}{c}{}&\multicolumn{1}{c} {$S^{n+3}$}&\multicolumn{1}{c} {$\mo{r}{n+2}$}\\
       \cline{2-3}
   $C_{\eta}:\s{n}$ & $T_4(a,0)$ &1  \\
      \quad ${\s{n+2}}$ & 0 &0\\
        \cline{2-3}
   $\mo{s}{n}$ & 1 &0\\
        \cline{2-3}
     \end{tabular}\end{footnotesize} $ ~~;
     \end{itemize}

   \begin{itemize}
     \item [(7)]  $\begin{footnotesize}\begin{tabular}{c|c|}
   \multicolumn{1}{c}{}&\multicolumn{1}{c} {$S^{n+3}$}\\
       \cline{2-2}
     $S^{n}$ & $T_8(a,b)$ \\
       \cline{2-2}
      $S^{n+2}$ & $1$\\
        \cline{2-2}
     \end{tabular}\end{footnotesize} $~~;~~$\begin{footnotesize}\begin{tabular}{c|c|c|}
   \multicolumn{1}{c}{}&\multicolumn{1}{c} {$S^{n+3}$}&\multicolumn{1}{c} {$\mo{r}{n+2}$}\\
       \cline{2-3}
     $S^{n}$ & $T_8(a,0)$&1\\
       \cline{2-3}
      $S^{n+2}$ & $1$&0\\
        \cline{2-3}
     \end{tabular}\end{footnotesize} $

     \item []$\begin{footnotesize}\begin{tabular}{c|c|}
   \multicolumn{1}{c}{}&\multicolumn{1}{c} {$S^{n+3}$}\\
       \cline{2-2}
     $S^{n}$ & $T_8(a,0)$ \\
       \cline{2-2}
      $S^{n+2}$ & $1$\\
      \cline{2-2}
   $\mo{s}{n}$ & 1 \\
    \cline{2-2}
     \end{tabular}\end{footnotesize} $~~;~~$\begin{footnotesize}\begin{tabular}{c|c|c|}
   \multicolumn{1}{c}{}&\multicolumn{1}{c} {$S^{n+3}$}&\multicolumn{1}{c} {$\mo{r}{n+2}$}\\
       \cline{2-3}
     $S^{n}$ & $T_8(a,0)$&1\\
       \cline{2-3}
      $S^{n+2}$ & $1$&0\\
        \cline{2-3}
       $\mo{s}{n}$ & 1 &0\\
       \cline{2-3}
     \end{tabular}\end{footnotesize} $
     \\ \\ \\
  where $(a,b)\in \mathbb{Z}/8\times \mathbb{Z}/3$ such that $a\in \{0,1,2,3\}\subset \mathbb{Z}/8 $ and $(a,b)\neq (0,0)$ for $T_8$; $(a,b)\in \mathbb{Z}/4\times \mathbb{Z}/3$ such that $a\in \{0,1,2,3\}\subset \mathbb{Z}/4 $ and $(a,b)\neq (0,0)$ for $T_4$. $s, r \in \mathbb{N}_+$.
     \end{itemize}

\item [(III)]   $\begin{footnotesize}\begin{tabular}{c|c|}
   \multicolumn{1}{c}{}&\multicolumn{1}{c} {$S^{n+3}$}\\
       \cline{2-2}
     $S^{n}$ &  $T_8(a,b)$  \\
      \cline{2-2}
    \end{tabular}\end{footnotesize} $~;$\begin{footnotesize}\begin{tabular}{c|c|c|}
   \multicolumn{1}{c}{}&\multicolumn{1}{c} {$S^{n+3}$}&\multicolumn{1}{c} {$\mo{r}{n+2}$}\\
       \cline{2-3}
     $S^{n}$ &  $T_8(a,0)$  &1\\
      \cline{2-3}
    \end{tabular}\end{footnotesize} $~;$\begin{footnotesize}\begin{tabular}{c|c|}
   \multicolumn{1}{c}{}&\multicolumn{1}{c} {$S^{n+3}$}\\
       \cline{2-2}
     $S^{n}$ &  $T_8(a,0)$  \\
      \cline{2-2}
        $\mo{s}{n}$ & 1 \\
         \cline{2-2}
    \end{tabular}\end{footnotesize} $~;$\begin{footnotesize}\begin{tabular}{c|c|c|}
   \multicolumn{1}{c}{}&\multicolumn{1}{c} {$S^{n+3}$}&\multicolumn{1}{c} {$\mo{r}{n+2}$}\\
       \cline{2-3}
     $S^{n}$ &  $T_8(a,0)$  &1\\
      \cline{2-3}
       $\mo{s}{n}$ & 1 &0\\
       \cline{2-3}
    \end{tabular}\end{footnotesize} $~;
   where $(a,b)\in \mathbb{Z}/8\times \mathbb{Z}/3$ such that $(a,b)\neq (0,0)$. $s, r \in \mathbb{N}_+$.

\item [(IV)]
    $(1)~ \begin{footnotesize}\begin{tabular}{c|c|}
   \multicolumn{1}{c}{}&\multicolumn{1}{c} {$S^{n+2}$}\\
       \cline{2-2}
     $S^{n+1}$ & $1$ \\
      \cline{2-2}
    \end{tabular}\end{footnotesize}$ ~;~

    $(2)~\begin{footnotesize}\begin{tabular}{c|c|}
   \multicolumn{1}{c}{}&\multicolumn{1}{c} {$S^{n+3}$}\\
       \cline{2-2}
     $S^{n+2}$ & $1$ \\
      \cline{2-2}
    \end{tabular}\end{footnotesize}$~;~
     $\begin{footnotesize}\begin{tabular}{c|c|}
   \multicolumn{1}{c}{}&\multicolumn{1}{c} {$S^{n+3}$}\\
       \cline{2-2}
     $S^{n+2}$ & $1$ \\
      \cline{2-2}
       $\mo{s}{n}$ & 1\\
       \cline{2-2}
    \end{tabular}\end{footnotesize}$~;

    $(3)~\begin{footnotesize}\begin{tabular}{c|c|}
   \multicolumn{1}{c}{}&\multicolumn{1}{c} {$S^{n+2}$}\\
       \cline{2-2}
     $S^{n}$ & $1$ \\
      \cline{2-2}
    \end{tabular}\end{footnotesize}$~;~$\begin{footnotesize}\begin{tabular}{c|c|c|}
   \multicolumn{1}{c}{}&\multicolumn{1}{c} {$S^{n+2}$}&\multicolumn{1}{c} {$\mo{r}{n+2}$}\\
       \cline{2-3}
     $S^{n}$ & $1$&1 \\
      \cline{2-3}
    \end{tabular}\end{footnotesize} $~;

    $(4)~\begin{footnotesize}\begin{tabular}{c|c|}
   \multicolumn{1}{c}{}&\multicolumn{1}{c} {$S^{n+3}$}\\
       \cline{2-2}
     $S^{n+1}$ & $1$ \\
      \cline{2-2}
    \end{tabular}\end{footnotesize} $~;~
$\begin{footnotesize}\begin{tabular}{c|c|}
   \multicolumn{1}{c}{}&\multicolumn{1}{c} {$S^{n+3}$}\\
       \cline{2-2}
     $S^{n+1}$ & $1$ \\
      \cline{2-2}
     $\mo{s}{n}$  & $1$ \\
      \cline{2-2}
    \end{tabular}\end{footnotesize} $ .
\\ \\ \\ where $r , s  \in \mathbb{N}_+$.
\end{itemize}

 Through a detailed check by admissible transformations of matrix problem  $(\mathscr{A}', \mathcal{G}')$, it follows the following

\begin{theorem}\label{theorem6.3}
All indecomposable isomorphic classes of $(\mathscr{A}', \mathcal{G}')$ are given by the following list
\begin{itemize}
  \item [$(I)$] $\begin{footnotesize}\begin{tabular}{c|c|}
   \multicolumn{1}{c}{}&\multicolumn{1}{c} {$S^{n+3}$}\\
       \cline{2-2}
     $S^{n+2}$ & $1$ \\
       \cline{2-2}
      $C_{\eta}:\s{n}$ & $v$\\
      \quad ${\s{n+2}}$ & $0$ \\
        \cline{2-2}
          \multicolumn{2}{c}{} \\
     \multicolumn{1}{c}{} & \multicolumn{1}{c} {$X(\eta v \eta)$}\\
     \end{tabular}\end{footnotesize}$ ~~;
     ~~$\begin{footnotesize}\begin{tabular}{c|c|c|}
   \multicolumn{1}{c}{}&\multicolumn{1}{c} {$S^{n+3}$}&\multicolumn{1}{c} {$\mo{r}{n+2}$}\\
       \cline{2-3}
     $S^{n+2}$ & $1$ & $0$\\
       \cline{2-3}
      $C_{\eta}:\s{n}$ & $3$& $1$\\
      \quad ${\s{n+2}}$ & $0$ & $0$\\
        \cline{2-3}
         \multicolumn{3}{c}{} \\
     \multicolumn{1}{c}{} & \multicolumn{2}{c} {$X(\eta 3 \eta)^r$}\\
     \end{tabular}\end{footnotesize}$~;

     $\begin{footnotesize}\begin{tabular}{c|c|}
   \multicolumn{1}{c}{}&\multicolumn{1}{c} {$S^{n+3}$}\\
       \cline{2-2}
     $S^{n+2}$ & $1$ \\
       \cline{2-2}
      $C_{\eta}:\s{n}$ & $3$\\
      \quad ${\s{n+2}}$ & $0$ \\
        \cline{2-2}
       $\mo{s}{n}$ & $1$\\
        \cline{2-2}
         \multicolumn{2}{c}{} \\
     \multicolumn{1}{c}{} & \multicolumn{1}{c} {$X(\eta 3 \eta)_s$}\\
     \end{tabular}\end{footnotesize}$ ~~;
     ~~$\begin{footnotesize}\begin{tabular}{c|c|c|}
   \multicolumn{1}{c}{}&\multicolumn{1}{c} {$S^{n+3}$}&\multicolumn{1}{c} {$\mo{r}{n+2}$}\\
       \cline{2-3}
     $S^{n+2}$ & $1$ & $0$\\
       \cline{2-3}
      $C_{\eta}:\s{n}$ & $3$ & $1$\\
      \quad ${\s{n+2}}$ & $0$ & $0$\\
        \cline{2-3}
         $\mo{s}{n}$ & $1$& $0$\\
        \cline{2-3}
         \multicolumn{3}{c}{} \\
     \multicolumn{1}{c}{} & \multicolumn{2}{c} {$X(\eta 3 \eta)_s^r$}\\
     \end{tabular}\end{footnotesize}$~~,
     \\ \\ \\ where $v\in\{1,2,3\}\subset \mathbb{Z}/12.$
  \item [$(II)$]
   \begin{itemize}
     \item [$(1)$] $\begin{footnotesize}\begin{tabular}{c|c|c|}
   \multicolumn{1}{c}{}&\multicolumn{1}{c} {$S^{n+2}$}&\multicolumn{1}{c} {$S^{n+3}$}\\
       \cline{2-3}
     $S^{n}$ & $1$& $v$ \\
        \cline{2-3}
      $S^{n+1}$ & $0$& $1$ \\
       \cline{2-3}
         \multicolumn{3}{c}{} \\
     \multicolumn{1}{c}{} & \multicolumn{2}{c} {$X(\eta\eta v\eta\eta)$}\\
     \end{tabular}\end{footnotesize}$ ~~;~~   $\begin{footnotesize}\begin{tabular}{c|c|c|c|}
   \multicolumn{1}{c}{}&\multicolumn{1}{c} {$S^{n+2}$}&\multicolumn{1}{c} {$S^{n+3}$}&\multicolumn{1}{c} {$\mo{r}{n+2}$}\\
       \cline{2-4}
     $S^{n}$ & $1$& $v_{1}$ & $1$\\
        \cline{2-4}
      $S^{n+1}$ & $0$& $1$ & $0$\\
       \cline{2-4}
        \multicolumn{4}{c}{} \\
     \multicolumn{1}{c}{} & \multicolumn{3}{c} {$X(\eta\eta v_1\eta\eta)^r$}\\
     \end{tabular}\end{footnotesize}$~~;

     \item [] $\begin{footnotesize}\begin{tabular}{c|c|c|}
   \multicolumn{1}{c}{}&\multicolumn{1}{c} {$S^{n+2}$}&\multicolumn{1}{c} {$S^{n+3}$}\\
       \cline{2-3}
     $S^{n}$ & $1$& $v_{1}$ \\
        \cline{2-3}
      $S^{n+1}$ & $0$& $1$ \\
       \cline{2-3}
        $\mo{s}{n}$ & $0$& $1$ \\
       \cline{2-3}
        \multicolumn{3}{c}{} \\
     \multicolumn{1}{c}{} & \multicolumn{2}{c} {$X(\eta\eta v_1\eta\eta)_s$}\\
     \end{tabular}\end{footnotesize}$ ~~;~~   $\begin{footnotesize}\begin{tabular}{c|c|c|c|}
   \multicolumn{1}{c}{}&\multicolumn{1}{c} {$S^{n+2}$}&\multicolumn{1}{c} {$S^{n+3}$}&\multicolumn{1}{c} {$\mo{r}{n+2}$}\\
       \cline{2-4}
     $S^{n}$ & $1$& $v_{1}$ & $1$\\
        \cline{2-4}
      $S^{n+1}$ & $0$& $1$ & $0$\\
       \cline{2-4}
       $\mo{s}{n}$ & $0$& $1$& $0$ \\
      \cline{2-4}
        \multicolumn{4}{c}{} \\
     \multicolumn{1}{c}{} & \multicolumn{3}{c} {$X(\eta\eta v_1\eta\eta)^r_s$}\\
     \end{tabular}\end{footnotesize}$~~;
    \end{itemize}

   \begin{itemize}
     \item [$(2)$] $\begin{footnotesize}\begin{tabular}{c|c|c|}
   \multicolumn{1}{c}{}&\multicolumn{1}{c} {$S^{n+2}$}&\multicolumn{1}{c} {$S^{n+3}$}\\
       \cline{2-3}
     $S^{n}$ & $1$& $v$ \\
        \cline{2-3}
      $S^{n+2}$ & $0$& $1$ \\
       \cline{2-3}
        \multicolumn{3}{c}{} \\
     \multicolumn{1}{c}{} & \multicolumn{2}{c} {$X(\eta\eta v\eta)$}\\
     \end{tabular}\end{footnotesize}$ ~~;~~   $\begin{footnotesize}\begin{tabular}{c|c|c|c|}
   \multicolumn{1}{c}{}&\multicolumn{1}{c} {$S^{n+2}$}&\multicolumn{1}{c} {$S^{n+3}$}&\multicolumn{1}{c} {$\mo{r}{n+2}$}\\
       \cline{2-4}
     $S^{n}$ & $1$& $v_{1}$ & $1$\\
        \cline{2-4}
      $S^{n+2}$ & $0$& $1$ & $0$\\
       \cline{2-4}
        \multicolumn{4}{c}{} \\
     \multicolumn{1}{c}{} & \multicolumn{3}{c} {$X(\eta\eta v_1\eta)^r$}\\
     \end{tabular}\end{footnotesize}$~~;

     \item [] $\begin{footnotesize}\begin{tabular}{c|c|c|}
   \multicolumn{1}{c}{}&\multicolumn{1}{c} {$S^{n+2}$}&\multicolumn{1}{c} {$S^{n+3}$}\\
       \cline{2-3}
     $S^{n}$ & $1$& $v_{1}$ \\
        \cline{2-3}
      $S^{n+2}$ & $0$& $1$ \\
       \cline{2-3}
        $\mo{s}{n}$ & $0$& $1$ \\
       \cline{2-3}
          \multicolumn{3}{c}{} \\
     \multicolumn{1}{c}{} & \multicolumn{2}{c} {$X(\eta\eta v_1\eta)_s$}\\
     \end{tabular}\end{footnotesize}$ ~~;~~   $\begin{footnotesize}\begin{tabular}{c|c|c|c|}
   \multicolumn{1}{c}{}&\multicolumn{1}{c} {$S^{n+2}$}&\multicolumn{1}{c} {$S^{n+3}$}&\multicolumn{1}{c} {$\mo{r}{n+2}$}\\
       \cline{2-4}
     $S^{n}$ & $1$& $v_{1}$ & $1$\\
        \cline{2-4}
      $S^{n+2}$ & $0$& $1$ & $0$\\
       \cline{2-4}
       $\mo{s}{n}$ & $0$& $1$& $0$ \\
      \cline{2-4}
        \multicolumn{4}{c}{} \\
     \multicolumn{1}{c}{} & \multicolumn{3}{c} {$X(\eta\eta v_1\eta)^r_s$}\\
     \end{tabular}\end{footnotesize}$~~;
    \end{itemize}

   \begin{itemize}
     \item [$(3)$] $\begin{footnotesize}\begin{tabular}{c|c|}
   \multicolumn{1}{c}{}&\multicolumn{1}{c} {$S^{n+3}$}\\
       \cline{2-2}
     $S^{n+1}$ & $1$ \\
       \cline{2-2}
      $C_{\eta}:\s{n}$ & $v$\\
      \quad ${\s{n+2}}$ & $0$ \\
        \cline{2-2}
           \multicolumn{2}{c}{} \\
     \multicolumn{1}{c}{} & \multicolumn{1}{c} {$X(\eta v\eta\eta)$}\\
     \end{tabular}\end{footnotesize}$ ~~;~~$\begin{footnotesize}\begin{tabular}{c|c|c|}
   \multicolumn{1}{c}{}&\multicolumn{1}{c} {$S^{n+3}$}&\multicolumn{1}{c} {$\mo{r}{n+2}$}\\
       \cline{2-3}
     $S^{n+1}$ & $1$ & $0$\\
       \cline{2-3}
      $C_{\eta}:\s{n}$ & $v_{1}$& $1$\\
      \quad ${\s{n+2}}$ & $0$ & $0$\\
        \cline{2-3}
              \multicolumn{3}{c}{} \\
     \multicolumn{1}{c}{} & \multicolumn{2}{c} {$X(\eta v_1\eta\eta)^r$}\\
     \end{tabular}\end{footnotesize}$~;

     \item [] $\begin{footnotesize}\begin{tabular}{c|c|}
   \multicolumn{1}{c}{}&\multicolumn{1}{c} {$S^{n+3}$}\\
       \cline{2-2}
     $S^{n+2}$ & $1$ \\
       \cline{2-2}
      $C_{\eta}:\s{n}$ & $v_{1}$\\
      \quad ${\s{n+2}}$ & $0$ \\
        \cline{2-2}
       $\mo{s}{n}$ & $1$\\
        \cline{2-2}
           \multicolumn{2}{c}{} \\
     \multicolumn{1}{c}{} & \multicolumn{1}{c} {$X(\eta v_1\eta\eta)_s$}\\
     \end{tabular}\end{footnotesize}$ ~~;~~$\begin{footnotesize}\begin{tabular}{c|c|c|}
   \multicolumn{1}{c}{}&\multicolumn{1}{c} {$S^{n+3}$}&\multicolumn{1}{c} {$\mo{r}{n+2}$}\\
       \cline{2-3}
     $S^{n+2}$ & $1$ & $0$\\
       \cline{2-3}
      $C_{\eta}:\s{n}$ & $v_{1}$& $1$\\
      \quad ${\s{n+2}}$ & $0$ & $0$\\
        \cline{2-3}
         $\mo{s}{n}$ & $1$& $0$\\
        \cline{2-3}
              \multicolumn{3}{c}{} \\
     \multicolumn{1}{c}{} & \multicolumn{2}{c} {$X(\eta v_1\eta\eta)^r_s$}\\
     \end{tabular}\end{footnotesize}$~;
    \end{itemize}

     \begin{itemize}
       \item [$(4)$]$\begin{footnotesize}\begin{tabular}{c|c|c|}
      \multicolumn{1}{c}{}&\multicolumn{1}{c} {$S^{n+2}$}&\multicolumn{1}{c} {$S^{n+3}$}\\
       \cline{2-3}
     $S^{n}$ & $1$& $v$  \\
      \cline{2-3}
          \multicolumn{3}{c}{} \\
     \multicolumn{1}{c}{} & \multicolumn{2}{c} {$X(\eta\eta v)$}\\
     \end{tabular}\end{footnotesize}$~~;~~$\begin{footnotesize}\begin{tabular}{c|c|c|c|}
      \multicolumn{1}{c}{}&\multicolumn{1}{c} {$S^{n+2}$}&\multicolumn{1}{c} {$S^{n+3}$}&\multicolumn{1}{c} {$\mo{r}{n+2}$}\\
       \cline{2-4}
     $S^{n}$ & $1$& $v_{1}$ & $1$ \\
      \cline{2-4}
        \multicolumn{4}{c}{} \\
     \multicolumn{1}{c}{} & \multicolumn{3}{c} {$X(\eta\eta v_1)^r$}\\
     \end{tabular}\end{footnotesize}$~~;

       \item []$\begin{footnotesize}\begin{tabular}{c|c|c|}
      \multicolumn{1}{c}{}&\multicolumn{1}{c} {$S^{n+2}$}&\multicolumn{1}{c} {$S^{n+3}$}\\
       \cline{2-3}
     $S^{n}$ & $1$& $v_{1}$  \\
      \cline{2-3}
     $\mo{s}{n}$ & $0$& $1$ \\
      \cline{2-3}
        \multicolumn{3}{c}{} \\
     \multicolumn{1}{c}{} & \multicolumn{2}{c} {$X(\eta\eta v_1)_s$}\\
     \end{tabular}\end{footnotesize}$~~;~~$\begin{footnotesize}\begin{tabular}{c|c|c|c|}
      \multicolumn{1}{c}{}&\multicolumn{1}{c} {$S^{n+2}$}&\multicolumn{1}{c} {$S^{n+3}$}&\multicolumn{1}{c} {$\mo{r}{n+2}$}\\
       \cline{2-4}
     $S^{n}$ & $1$& $v_{1}$ & $1$ \\
      \cline{2-4}
        $\mo{s}{n}$ & $0$& $1$ & $0$\\
          \cline{2-4}
            \multicolumn{4}{c}{} \\
     \multicolumn{1}{c}{} & \multicolumn{3}{c} {$X(\eta\eta v_1)^r_s$}\\
     \end{tabular}\end{footnotesize}$~~;
     \end{itemize}

     \begin{itemize}
       \item [$(5)$]$\begin{footnotesize}\begin{tabular}{c|c|}
   \multicolumn{1}{c}{}&\multicolumn{1}{c} {$S^{n+3}$}\\
       \cline{2-2}
     $S^{n}$ &  $v$  \\
       \cline{2-2}
      $S^{n+1}$ & $1$\\
        \cline{2-2}
          \multicolumn{2}{c}{} \\
     \multicolumn{1}{c}{} & \multicolumn{1}{c} {$X( v \eta\eta)$}\\
     \end{tabular}\end{footnotesize} $~~;~~ $\begin{footnotesize}\begin{tabular}{c|c|c|}
   \multicolumn{1}{c}{}&\multicolumn{1}{c} {$S^{n+3}$}&\multicolumn{1}{c} {$\mo{r}{n+2}$}\\
       \cline{2-3}
     $S^{n}$ & $v_{1}$  & $1$\\
       \cline{2-3}
      $S^{n+1}$ & $1$& $0$\\
        \cline{2-3}
           \multicolumn{3}{c}{} \\
     \multicolumn{1}{c}{} & \multicolumn{2}{c} {$X( v_1 \eta\eta)^r$}\\
     \end{tabular}\end{footnotesize} $~~;

       \item []$\begin{footnotesize}\begin{tabular}{c|c|}
   \multicolumn{1}{c}{}&\multicolumn{1}{c} {$S^{n+3}$}\\
       \cline{2-2}
     $S^{n}$ &  $v$  \\
       \cline{2-2}
      $S^{n+1}$ & $1$\\
        \cline{2-2}
      $\mo{s}{n}$ & $1$\\
        \cline{2-2}
          \multicolumn{2}{c}{} \\
     \multicolumn{1}{c}{} & \multicolumn{1}{c} {$X( v_1 \eta\eta)_s$}\\
     \end{tabular}\end{footnotesize} $~~;~~ $\begin{footnotesize}\begin{tabular}{c|c|c|}
   \multicolumn{1}{c}{}&\multicolumn{1}{c} {$S^{n+3}$}&\multicolumn{1}{c} {$\mo{r}{n+2}$}\\
       \cline{2-3}
     $S^{n}$ & $v_{1}$  & $1$\\
       \cline{2-3}
      $S^{n+1}$ & $1$& $0$\\
        \cline{2-3}
       $\mo{s}{n}$& $1$& $0$\\
        \cline{2-3}
           \multicolumn{3}{c}{} \\
     \multicolumn{1}{c}{} & \multicolumn{2}{c} {$X( v_1 \eta\eta)^r_s$}\\
     \end{tabular}\end{footnotesize} $~~;
     \end{itemize}

     \begin{itemize}
       \item [$(6)$] $ \begin{footnotesize}\begin{tabular}{c|c|}
   \multicolumn{1}{c}{}&\multicolumn{1}{c} {$S^{n+3}$}\\
       \cline{2-2}
   $C_{\eta}:\s{n}$ & $v$   \\
      \quad ${\s{n+2}}$ & $0$ \\
        \cline{2-2}
             \multicolumn{2}{c}{} \\
     \multicolumn{1}{c}{} & \multicolumn{1}{c} {$X(\eta v )$}\\
     \end{tabular}\end{footnotesize} $ ~~;~~$ \begin{footnotesize}\begin{tabular}{c|c|c|}
   \multicolumn{1}{c}{}&\multicolumn{1}{c} {$S^{n+3}$}&\multicolumn{1}{c} {$\mo{r}{n+2}$}\\
       \cline{2-3}
   $C_{\eta}:\s{n}$ & $v_{1}$ & $1$  \\
      \quad ${\s{n+2}}$ & $0$ & $0$\\
        \cline{2-3}
          \multicolumn{3}{c}{} \\
     \multicolumn{1}{c}{} & \multicolumn{2}{c} {$X( \eta v_1 )^r$}\\
     \end{tabular}\end{footnotesize} $ ~~;

       \item [] $ \begin{footnotesize}\begin{tabular}{c|c|}
   \multicolumn{1}{c}{}&\multicolumn{1}{c} {$S^{n+3}$}\\
       \cline{2-2}
   $C_{\eta}:\s{n}$ & $v_{1}$   \\
      \quad ${\s{n+2}}$ & $0$ \\
        \cline{2-2}
$\mo{s}{n}$ & $1$ \\
        \cline{2-2}
         \multicolumn{2}{c}{} \\
     \multicolumn{1}{c}{} & \multicolumn{1}{c} {$X(\eta v_1 )_s$}\\
     \end{tabular}\end{footnotesize} $ ~~;~~$ \begin{footnotesize}\begin{tabular}{c|c|c|}
   \multicolumn{1}{c}{}&\multicolumn{1}{c} {$S^{n+3}$}&\multicolumn{1}{c} {$\mo{r}{n+2}$}\\
       \cline{2-3}
   $C_{\eta}:\s{n}$ & $v_{1}$ & $1$  \\
      \quad ${\s{n+2}}$ & $0$ & $0$\\
        \cline{2-3}
   $\mo{s}{n}$ & $1$ & $0$\\
        \cline{2-3}
          \multicolumn{3}{c}{} \\
     \multicolumn{1}{c}{} & \multicolumn{2}{c} {$X( \eta v_1 )^r_s$}\\
     \end{tabular}\end{footnotesize} $ ~~;
     \end{itemize}

   \begin{itemize}
     \item [$(7)$]  $\begin{footnotesize}\begin{tabular}{c|c|}
   \multicolumn{1}{c}{}&\multicolumn{1}{c} {$S^{n+3}$}\\
       \cline{2-2}
     $S^{n}$ & $v$ \\
       \cline{2-2}
      $S^{n+2}$ & $1$\\
        \cline{2-2}
        \multicolumn{2}{c}{} \\
     \multicolumn{1}{c}{} & \multicolumn{1}{c} {$X( v\eta )$}\\
     \end{tabular}\end{footnotesize} $~~;~~$\begin{footnotesize}\begin{tabular}{c|c|c|}
   \multicolumn{1}{c}{}&\multicolumn{1}{c} {$S^{n+3}$}&\multicolumn{1}{c} {$\mo{r}{n+2}$}\\
       \cline{2-3}
     $S^{n}$ & $v_{1}$& $1$\\
       \cline{2-3}
      $S^{n+2}$ & $1$& $0$\\
      \cline{2-3}
         \multicolumn{3}{c}{} \\
     \multicolumn{1}{c}{} & \multicolumn{2}{c} {$X( v_1\eta  )^r$}\\
     \end{tabular}\end{footnotesize} ~~;$

     \item []$\begin{footnotesize}\begin{tabular}{c|c|}
   \multicolumn{1}{c}{}&\multicolumn{1}{c} {$S^{n+3}$}\\
       \cline{2-2}
     $S^{n}$ & $v_{1}$ \\
       \cline{2-2}
      $S^{n+2}$ & $1$\\
      \cline{2-2}
   $\mo{s}{n}$ & $1$ \\
    \cline{2-2}
      \multicolumn{2}{c}{} \\
     \multicolumn{1}{c}{} & \multicolumn{1}{c} {$X( v_1\eta )_s$}\\
     \end{tabular}\end{footnotesize} $~~;~~$\begin{footnotesize}\begin{tabular}{c|c|c|}
   \multicolumn{1}{c}{}&\multicolumn{1}{c} {$S^{n+3}$}&\multicolumn{1}{c} {$\mo{r}{n+2}$}\\
       \cline{2-3}
     $S^{n}$ & $v_{1}$& $1$\\
       \cline{2-3}
      $S^{n+2}$ & $1$& $0$\\
        \cline{2-3}
       $\mo{s}{n}$ & $1$ & $0$\\
       \cline{2-3}
       \multicolumn{3}{c}{} \\
     \multicolumn{1}{c}{} & \multicolumn{2}{c} {$X( v_1\eta  )^r_s$}\\
     \end{tabular}\end{footnotesize} ~~,$
  \\ \\ \\
  where $v\in\{1,2,3,4,5,6\}\subset \mathbb{Z}/24$ or $\mathbb{Z}/12$ , $v_1\in\{3,6\}\subset \mathbb{Z}/24$ or $\mathbb{Z}/12$ and $r , s\in \mathbb{N}_+$.
  \end{itemize}

  \item [$(III)$] $\begin{footnotesize}\begin{tabular}{c|c|}
   \multicolumn{1}{c}{}&\multicolumn{1}{c} {$S^{n+3}$}\\
       \cline{2-2}
     $S^{n}$ &  $v$  \\
      \cline{2-2}
         \multicolumn{2}{c}{} \\
     \multicolumn{1}{c}{} & \multicolumn{1}{c} {$X( v )$}\\
    \end{tabular}\end{footnotesize} $~;~$\begin{footnotesize}\begin{tabular}{c|c|c|}
   \multicolumn{1}{c}{}&\multicolumn{1}{c} {$S^{n+3}$}&\multicolumn{1}{c} {$\mo{r}{n+2}$}\\
       \cline{2-3}
     $S^{n}$ &  $v_1$  & $1$\\
      \cline{2-3}
       \multicolumn{3}{c}{} \\
     \multicolumn{1}{c}{} & \multicolumn{2}{c} {$X( v_1 )^r$}\\
    \end{tabular}\end{footnotesize} $~;~$\begin{footnotesize}\begin{tabular}{c|c|}
   \multicolumn{1}{c}{}&\multicolumn{1}{c} {$S^{n+3}$}\\
       \cline{2-2}
     $S^{n}$ &  $v_1$  \\
      \cline{2-2}
        $\mo{s}{n}$ & $1$ \\
         \cline{2-2}
           \multicolumn{2}{c}{} \\
     \multicolumn{1}{c}{} & \multicolumn{1}{c} {$X( v_1 )_s$}\\
    \end{tabular}\end{footnotesize} $~;~$\begin{footnotesize}\begin{tabular}{c|c|c|}
   \multicolumn{1}{c}{}&\multicolumn{1}{c} {$S^{n+3}$}&\multicolumn{1}{c} {$\mo{r}{n+2}$}\\
       \cline{2-3}
     $S^{n}$ &  $v_1$  & $1$\\
      \cline{2-3}
       $\mo{s}{n}$ & $1$ & $0$\\
       \cline{2-3}
            \multicolumn{3}{c}{} \\
     \multicolumn{1}{c}{} & \multicolumn{2}{c} {$X( v_1 )^r_s$}\\
    \end{tabular}\end{footnotesize} $~;~
    \\ \\  \\
    where $v\in\{1,2,\cdots,12\}\subset\mathbb{Z}/24$ and $v_1\in\{3,6,9\}\subset\mathbb{Z}/24$. $r ,s \in \mathbb{N}_+$.

\item [$(IV)$]
   $(1)\begin{footnotesize}\begin{tabular}{c|c|}
   \multicolumn{1}{c}{}&\multicolumn{1}{c} {$S^{n+2}$}\\
       \cline{2-2}
     $S^{n+1}$ & $1$ \\
      \cline{2-2}
         \multicolumn{2}{c}{} \\
     \multicolumn{1}{c}{} & \multicolumn{1}{c} {$X(\eta_1 )$}\\
    \end{tabular}\end{footnotesize}$ ~;~

    $(2)~\begin{footnotesize}\begin{tabular}{c|c|}
   \multicolumn{1}{c}{}&\multicolumn{1}{c} {$S^{n+3}$}\\
       \cline{2-2}
     $S^{n+2}$ & $1$ \\
      \cline{2-2}
         \multicolumn{2}{c}{} \\
     \multicolumn{1}{c}{} & \multicolumn{1}{c} {$X(\eta_2 )$}\\
    \end{tabular}\end{footnotesize}$~;~
     $\begin{footnotesize}\begin{tabular}{c|c|}
   \multicolumn{1}{c}{}&\multicolumn{1}{c} {$S^{n+3}$}\\
       \cline{2-2}
     $S^{n+2}$ & $1$ \\
      \cline{2-2}
       $\mo{s}{n}$ & $1$\\
       \cline{2-2}
          \multicolumn{2}{c}{} \\
     \multicolumn{1}{c}{} & \multicolumn{1}{c} {$X(\eta_2 )_s$}\\
    \end{tabular}\end{footnotesize}~ ;$

    $(3)~\begin{footnotesize}\begin{tabular}{c|c|}
   \multicolumn{1}{c}{}&\multicolumn{1}{c} {$S^{n+2}$}\\
       \cline{2-2}
     $S^{n}$ & $1$ \\
      \cline{2-2}
        \multicolumn{2}{c}{} \\
     \multicolumn{1}{c}{} & \multicolumn{1}{c} {$X(\eta\eta )_{0}$}\\
    \end{tabular}\end{footnotesize}$~;~$\begin{footnotesize}\begin{tabular}{c|c|c|}
   \multicolumn{1}{c}{}&\multicolumn{1}{c} {$S^{n+2}$}&\multicolumn{1}{c} {$\mo{r}{n+2}$}\\
       \cline{2-3}
     $S^{n}$ & $1$& $1$ \\
      \cline{2-3}
        \multicolumn{2}{c}{} \\
     \multicolumn{1}{c}{} & \multicolumn{1}{c} {$X(\eta\eta )_{0}^r$}\\
    \end{tabular}\end{footnotesize} $~;

    $(4)~\begin{footnotesize}\begin{tabular}{c|c|}
   \multicolumn{1}{c}{}&\multicolumn{1}{c} {$S^{n+3}$}\\
       \cline{2-2}
     $S^{n+1}$ & $1$ \\
      \cline{2-2}
         \multicolumn{2}{c}{} \\
     \multicolumn{1}{c}{} & \multicolumn{1}{c} {$X(\eta\eta )_{1}$}\\
    \end{tabular}\end{footnotesize} $~;~
$\begin{footnotesize}\begin{tabular}{c|c|}
   \multicolumn{1}{c}{}&\multicolumn{1}{c} {$S^{n+3}$}\\
       \cline{2-2}
     $S^{n+1}$ & $1$ \\
      \cline{2-2}
     $\mo{s}{n}$  & $1$ \\
      \cline{2-2}
         \multicolumn{2}{c}{} \\
     \multicolumn{1}{c}{} & \multicolumn{1}{c} {$X(\eta\eta )_{1s}$}\\
    \end{tabular}\end{footnotesize}$~,
 \\ \\ \\ where $r,s\in \mathbb{N}_+$.
\end{itemize}

\end{theorem}

It is easy to recover the polyhedra from the matrices listed in Theorem \ref{theorem6.3}.
For example,    ~~$$\begin{footnotesize}\begin{tabular}{c|c|c|}
   \multicolumn{1}{c}{}&\multicolumn{1}{c} {$S^{n+3}$}&\multicolumn{1}{c} {$\mo{r}{n+2}$}\\
       \cline{2-3}
     $S^{n+2}$ & $1$ & $0$\\
       \cline{2-3}
      $C_{\eta}:\s{n}$ & $3$ & $1$\\
      \quad ${\s{n+2}}$ & $0$ & $0$\\
        \cline{2-3}
         $\mo{s}{n}$ & $1$& $0$\\
        \cline{2-3}
         \multicolumn{3}{c}{} \\
     \multicolumn{1}{c}{} & \multicolumn{2}{c} {$X(\eta 3 \eta)_s^r$}\\
     \end{tabular}\end{footnotesize}$$

 represents the  cone of the  map  $ S^{n+3}\vee \mo{r}{n+2}\rightarrow S^{n+2}\vee C_{\eta} \vee \mo{s}{n}$  corresponding to the matrix. Since
$$S^{n+2}\vee C_{\eta} \vee \mo{s}{n}=(S^{n+2}\vee S^{n}\vee S^{n})\cup_{i_2\eta}e^{n+2}\cup_{i_{3}3^s}e^{n+1}~,$$
we get that  the polyhedron corresponding to this matrix is
$$(S^{n+2}\vee S^{n}\vee S^{n}\vee S^{n+3})\cup_{i_2\eta}e^{n+2}\cup_{i_{3}3^r}e^{n+1}\cup_{i_{1}\eta\eta+i_{2}3+i_{3}1}e^{n+4}\cup_{i_{3}1+i_{4}3^r}e^{n+4}~,$$
where $i_{t}: X_{t}\hookrightarrow \bigvee_{j} X_{j}$ is the canonical inclusion of the summand.

Finally, from Corollary \ref{corollary 5.3}, Lemma \ref{lemma5.4} and Corollary \ref{corollary 5.5}, we obtain all the 2-torsion free indecomposable homotopy types of $\mathcal{A}\dag\mathcal{B}$, which completes the proof of Theorem \ref{maintheorem} (\textbf{Main theorem}).

\section{Concluding remarks}
\label{subsec:6.3}

In this paper, using the well known results about homotopy classes of maps between Moore spaces and suspended complex projective space and their compositions, we succeed in classifying indecomposable $\mathbf{F}^4_{n(2)}$ polyhedra. However the corresponding classification problems for the cases $\mathbf{F}^5_{n(2)}$ and $\mathbf{F}^6_{n(2)}$ are still open. We hope to return to this issue in the future publication. On the other hand, from the previous remark, it is crucial to understand globally a collection of spaces as a subcategory of  homotopy category of spaces.We will focus on this point in the future works.
\\
\\
\\

\textbf{Acknowledgements:} This work has been accepted for publication in SCIENCE CHINA Mathematics.
\\
\\
\\

\end{document}